\documentclass[twoside,reqno]{amsart}

\usepackage{amsmath}
\usepackage{amssymb,latexsym,amstext,amsthm}
\usepackage{verbatim} 
\usepackage{eucal} 
\usepackage{color}
\theoremstyle{plain}
\newtheorem{thm}[equation]{Theorem}
\newtheorem{pro}[equation]{Proposition}

\newtheorem{lem}[equation]{Lemma}

\theoremstyle{definition}
\newtheorem{exa}[equation]{Example}

\newtheorem{DEF}[equation]{Definition}
\newtheorem{rem}[equation]{Remark}

\renewcommand{\theequation}{\arabic{section}.\arabic{equation}}

\renewcommand{\theequation}{\thesubsection.\arabic{equation}}

\def\mod{\hbox{mod}}
\def\andd{\quad\hbox{and}\quad}

\def\op{\oplus}

\def\fm{(\cdot,\cdot)}
\def\a{\alpha}

\def\sub{\subseteq}

\def\lam{\lambda}
\def\Lam{\Lambda}

\def\1k{\frac{1}{k}}

\def\la{\langle}
\def\ra{\rangle}

\def\d{\delta}

\def\b{\beta}

\def\qed{\hfill$\Box$}

\def\sg{\sigma}

\def\i{{\mathcal I}}

\def\bq{{\bf q}}

\def\bbbz{{\mathbb Z}}

\def\bbbf{{\mathbb F}}
\def\bbbq{{\mathbb Q}}
\def\bbbe{{\mathbb E}}
\def\bbba{{\mathbb A}}

\def\aa{\mathcal A}

\def\tr{\hbox{tr}}

\def\ep{\epsilon}

\def\supp{\hbox{supp}}

\def\mfm{\mathfrak m}

\def\fsj{\mathfrak{fsj}}
\def\bbbn{\mathbb{N}}
\def\pp{\mathcal{P}}

\def\qedexa{\hfill$\diamondsuit$}

\begin{document}

\setcounter{page}{1} \setcounter{page}{1}

\author{Saeid Azam$\;^{1}$, Yoji Yoshii, Malihe Yousofzadeh$^{2}$}

\title{Jordan tori for a torsion free abelian group}


\address
{Department of Mathematics\\ University of Isfahan\\Isfahan, Iran,
P.O.Box 81745-163 and School of Mathematics, Institute for
Research in Fundamental Sciences (IPM), P.O. Box: 19395-5746,
Tehran, Iran.} \email{azam@sci.ui.ac.ir, saeidazam@yahoo.com.}

\address{Department of Mathematics Education\\ Iwate University\\
Ueda 3-18-33,
Morioka, Iwate\\
Japan 020-8550.}
\email{yoshii@iwate-u.ac.jp.}

\address{Department of Mathematics\\ University of Isfahan\\Isfahan, Iran,
P.O.Box 81745-163 and School of Mathematics, Institute for
Research in Fundamental Sciences (IPM), P.O. Box: 19395-5746,
Tehran, Iran.}\email{ma.yousofzadeh@sci.ui.ac.ir.}

\thanks{$\;^1$This research was in part supported by a reseach grant from IPM and partially carried out in IPM-Isfahan branch. 
The author also would like to thank the Center of Excellence for
Mathematics, University of Isfahan}


\thanks{$\;^2$This research was in part supported by a grant from IPM
(No. 91170415) and partially carried out in IPM-Isfahan branch.}



\bigskip

\begin{abstract}
We classify Jordan $G$-tori, where $G$ is any torsion-free abelian group.
Using the Zelmanov prime structure theorem,
such a class divides into three types,
namely, {the Hermitian type,  the Clifford type and the Albert type.}
We concretely describe Jordan $G$-tori of each type.
\end{abstract}

\maketitle


\setcounter{section}{-1}

\section{\bf Introduction}\label{introduction}\setcounter{equation}{0}
\markboth{S. Azam, Y. Yoshii, M. Yousofzadeh}{Jordan tori}
\begin{tiny}
•
\end{tiny}
It is a well-known fact that the concept of a ``$\bbbz^n$-torus'', is of great
importance in the context of classification of extended affine Lie
algebras. This concept was originally defined by Y. Yoshii in
\cite{Yo2}. With the appearance
of more general extensions of affine Kac-Moody Lie algebras such as,
locally extended affine Lie algebras and invariant affine {reflection algebras},
one naturally extends the concept of a $\bbbz^n$-torus to a $G$-torus
for
an abelian group $G$, {where for the algebras under consideration, $G$ is almost always torsion free.}
In this work we classify, in a descriptive manner, Jordan $G$-tori,
where $G$ is a torsion free abelian group.

{First we discuss associative $G$-tori, using the concept of  cocycles.}
Then we show that a Jordan $G$-torus is strongly prime,
and so one can use the Zelmanov prime structure theorem {\cite{MZ}}.
Thus,
such a class divides into three types,
the Hermitian type,  the Clifford type and the Albert type.
We classify each type
using the result of associative $G$-tori
and similar methods in \cite{Yo1}.

{The paper is organized as follows. In Section 1, we provide preliminary concepts, including
direct limits and direct unions, pointed reflection subspaces and (involutorial) associative $G$-tori. In Section 2, using a direct union approach, we show that a Jordan $G$-tori $J$ of Hermitian type has one of involution, plus or extension types (see Definition \ref{nolable}) and that $J$ is a direct union of Jordan tori of Hermitian type, where $J$ and its direct union components have the same involution, plus or extension type, see Theorem \ref{92-1}.
In Section 3, we show that a Jordan $G$-torus
$J$ of Clifford type with
support $S$ and central grading group $\Gamma$, is graded isomorphic to a Clifford $G$-torus $J(S,\Gamma,\{a_\ep\}_{\ep\in I})$, introduced explicitly in Example
\ref{clifford2}, for some nonempty index set $I$ and choices of $a_\ep\in\bbbf^\times$, $\ep\in I$, see Theorem \ref{clif}.
In Section 4, the final section, we first fully characterize associative $G$-tori of central degree 3. Then for two subgroups $\Delta$ and $\Gamma$
of $G$ satisfying $3G\subsetneq \Gamma\sub\Delta\sub G,$ $\dim_{\bbbz_3}(G/\Gamma)=3,$ and $\dim(\Delta/\Gamma)=2,$
we associate to the triple $(G,\Delta,\Gamma)$, a Jordan algebra $\bbba_t$ which turns out to be a Jordan $G$-torus of Albert type, called
an {\it Albert $G$-torus} associated to the triple $(G,\Delta,\Gamma)$, see Example \ref{yoshii-6.8-2}. Then we proceed with showing that
given a Jordan $G$-torus $J$ of Albert type with central grading
group $\Gamma$, there exists a subgroup $\Delta$ of $G$ such that the groups $G$, $\Delta$ and $\Gamma$ satisfy the above interactions and that
 $J$ is graded isomorphic to the Albert $G$-torus $\bbba_t$, constructed from the triple $(G,\Delta,\Gamma)$, see Theorem \ref{yoshii-thm-6.16}.}

\section{\bf Preliminaries}\label{preliminaries}\setcounter{equation}{0}
Throughout this work, $\bbbf$ is a filed {of characteristic zero} and $G$ is an abelian group. All algebras assumed over $\bbbf$ and are unital,
{unless otherwise mentioned.}
For a subset $X$ of an abelian group, by $\la X\ra$, we mean the subgroup generated by $X$.
In a graded algebra we speak of invertible (homogenous) elements, whenever this notion is defined.
The support of a $G$-graded algebra $T$, denoted $\supp(T)$, is by definition the set
of those elements of $G$ for which the corresponding homogenous space is nonzero.
For a set $X$, we denote by ${\mathcal M}_X$ the class of all finite
subsets of $X$. For an associative algebra $\aa$, we denote the corresponding plus algebra {by} $\aa^+$;
namely, $\aa^+$ has $\aa$ as its ground vector space, with Jordan product $a\circ b:=\frac{1}{2}(ab+ba)$.
If $\aa$ is equipped with an involution $\theta$ {(a period $2$ antiautomorphism)}, then $H(\aa,\theta):=\{a\in\aa\mid \theta(a)=a\}$ is a subalgebra of $\aa^+$.
{The field of rational numbers will be denoted by $\bbbq$.} {To indicate that a proof is finished, we
put the symbol $\Box$, and to indicate that an example is concluded, we put the symbol $\diamondsuit$. We refer the reader to {\cite{ZSSS}} for some terminologies
on nonassociative algebras used in the sequel, such as prime, strongly prime, degree, Jordan domain, etc.}

\subsection{A brief review of direct limits and direct unions} A set $I$ together with a partially ordering $\preceq,$ referred  to  $(I,\preceq),$ is called  a {\it directed set} if for each two elements $i,j\in I,$ there is $t\in I$ with $i\preceq t$ and $j\preceq t.$  Suppose that $\mathcal{C}$ is a category and $(I,\preceq)$  is a directed set. A family  $\{C_i\mid i\in I\}$ of objects of $\mathcal{C}$ together with a family $\{f_{i,j}\mid i,j\in I;\; i\preceq j\}$ of morphisms $f_{i,j}$ of $C_i$ to $C_j$ ($i,j\in I,$ $i\preceq j$)  is called a {\it direct system} in $\mathcal{C}$ if  for every pair
$(i, j)$ with $i \preceq j,$ $ f_{ii} = 1_{C_i}$ and $f_{k,i} = f_{k,j} \circ  f_{j,i}$ for $i\preceq j \preceq k.$ A {\it direct limit} of the
direct system $(\{C_i\}_{i\in I},\{f_{i,j}\}_{i\preceq j})$ is an object $C$ together with  morphisms
$\varphi_i : C_i\longrightarrow C$ $(i\in I)$ satisfying the following two conditions:
\begin{itemize}
\item $\varphi_i = \varphi_j\circ f_{i,j}$  for  $i,j\in I$ with  $i\preceq j,$
\item  for any other subject $D$ and morphisms $\psi_i$ ($i\in I$) from $C_i$ to $D$  with  $\psi_i = \psi_j\circ f_{i,j}$ for  $i,j\in I$ with  $i\preceq j,$   there exists a unique morphism $\psi $ from $C$ to $D$ such that $\psi\circ \varphi_i=\psi_i$ for $i\in I.$
\end{itemize}
If a direct limit of  a direct system $(\{C_i\}_{i\in I},\{f_{i,j}\}_{i\preceq j})$ in a category $\mathcal{C}$ exists, it is unique up to equivalence, so we refer to  as  {\it the direct limit} and denote it by $\underrightarrow{lim}C_i.$ Suppose that $C$ is the direct limit of a direct system  $(\{C_i\}_{i\in I},\{f_{i,j}\}_{i\preceq j})$ in a  concrete category $\mathcal{C}$ such that each $C_i$ is a subset of $C$ and for $i,j\in I$ with $i\preceq j,$ $f_{i,j}$ is the inclusion map, then we say $C$ is the {\it direct union} of $(\{C_i\}_{i\in I},\{f_{i,j}\}_{i\preceq j})$ if $C=\cup_{i\in I}C_i.$

\subsection{\bf Pointed reflection
subspaces}\label{pointed}\setcounter{equation}{0}

In this subsection we recall the notion of a reflection subspace and
record, in terms of a direct union point of view, certain properties
of reflection subspaces which will be needed in the sequel.

\begin{DEF}\label{b1}
{\em A {\it symmetric reflection subspace} of an additive abelian
group $G$ is a subset $S$ of $G$ satisfying $\la S\ra= G$ and
$S-2S\sub S$. A symmetric reflection subspace is called a {\it
pointed reflection subspace} (PRS for short) if $0\in S$. For details
on symmetric reflection subspaces, we refer the interested reader to
{\cite{Lo} and \cite{AYY}}.}\end{DEF}

If the group $G$ is free abelian of finite rank, a symmetric
reflection subspace in $G$ is also called a {\it translated semilattice} in $G$.
In this case a pointed reflection subspace is called a {\it
semilattice}. A non-trivial interesting feature of
semilattices is that any semilattice in $G$ contains a $\bbbz$-basis
of $G$ {(see \cite[Proposition, II.1.11]{AABGP})}.

The following lemma, {whose proof is straightforward}, gives a characterization of a PRS in terms of its
{finitely generated} pointed reflection subspaces.

\begin{lem}\label{b2}
(i) Let $S$ be a PRS in $G$. Then the {following} hold:

(a) For $T\sub S$, $S_T:=S\cap\la T\ra$ is a PRS in $\la T\ra$. In
particular, if $G$ is torsion free and $T$ is finite then $S_T$ is a
semilattice in $\la T\ra$.

(b) $S$ is the direct union of $\{S_T\}_{T\in{\mathcal M}_S}$, where ${\mathcal M}_S$ is directed
under inclusion.

(ii) Let $\mathcal S$ be a family of subsets of $G$ such that via the
inclusion $\mathcal S$ is a directed set, and that each element of
$\mathcal S$ is a PRS in its $\bbbz$-span in $G$. If
$G=\cup_{S\in{\mathcal S}}\la S\ra$, then the direct union of
$\{S\}_{S\in{\mathcal S}}$ is a PRS in $G$.
\end{lem}

\subsection{\bf $G$-tori}\label{tori}\setcounter{equation}{0}

In this subsection we study $G$-tori, where $G$ is assumed to be a
{\it torsion free} abelian group. Since {$G$ can be naturally imbedded in} $G\otimes_\bbbz
\bbbq$, we can make sense of $\sg/n$ for $\sg\in G$ and
$n\in\bbbz\setminus\{0\}$. {{We should recall that since $G$ is torsion free it is an ordered group in
the sense of \cite[page 94]{La}.}

\begin{DEF}\label{day-2}{\em
(\cite[Definition 3.1]{Yo1}) A $G$-graded algebra $J=\sum_{\sg\in
G}J^\sg$ satisfying conditions
\begin{itemize}
\item[(T1)] $G=\la \sg\in G\mid J^\sg\not=0\ra$,
\item[(T2)] all nonzero homogeneous elements of $J$ are
    invertible,
\item[(T3)] $\dim_\bbbf(J^\sg)\leq 1$ for all $\sg\in G$,
\end{itemize}
is called a {\it $G$-torus}. It is called of {{\it strong type}, if $J$ is {\it strongly graded}, namely
$J^\sg    J^\tau=J^{\sg+\tau}$ for all $\sg,\tau\in G$.
The $G$-torus $J$ is called an {\it associative} or a {\it Jordan
$G$-torus}, if $J$ is associative or Jordan, respectively.}}
\end{DEF}

\begin{lem}\label{direct-limit} Suppose that $G$ is an abelian group and $\Gamma$ is {a nonempty index set.}
Suppose that $\{G_\gamma\mid \gamma\in\Gamma\}$ is a class of subgroups of $G$ such  that $G=\cup_{\gamma\in\Gamma}G_\gamma.$
{Consider $\Gamma$ as a directed set whose ordering ``$\preccurlyeq$" is defined by}
$$\gamma\preccurlyeq\eta\hbox{ if $G_\gamma$ is a subgroup of $G_\eta$ } (\gamma,\eta\in\Gamma).  $$ If $(\{\aa_\gamma\},\{\varphi_{\gamma,\eta}\})$ is a direct system of associative algebras and algebra homomorphisms with direct limit $(\aa,\{\varphi_\gamma\})$ such that
\begin{itemize}
\item each $\aa_\gamma$ is equipped with a $G$-grading $\aa_\gamma=\op_{g\in G}(\aa_\gamma)^g$ with $\supp(\aa_\gamma)=G_\gamma$ and $\dim((\aa_\gamma)^g)\leq 1,$ for all $g\in G.$

\item each $\varphi_{\gamma,\eta}$ is a $G$-graded homomorphism,

\item each $\varphi_\gamma$ is monomorphism,
\end{itemize}
then $\aa$ as an algebra  is equipped with a  $G$-grading $\aa=\op_{g\in G}\aa^g$  with $\supp(\aa)=G$ and $\dim(\aa^g)=1$ for all $g\in G.$ Moreover if each $\aa_\gamma$ is an associative $G_\gamma$-torus, then  $\aa$ is an associative $G$-torus.
\end{lem}

\begin{proof}
It is easy to see.
\end{proof}

\begin{lem}\label{b3} Let $T$ be a Jordan or an associative $G$-torus. Let
$\mathcal M$ be the set of all finite subsets of $\supp(T)$
containing a fixed finite subset ${\mathfrak m}_0$ of $\supp(T)$, and
consider $\mathcal M$ as a directed set via inclusion. For
$\mathfrak{m}\in\mathcal M$, let $G_\mfm:=\la \mfm\ra$ and
$T_{\mathfrak m}:=\sum_{\sg\in G_\mfm}T^\sg$. Then we have the
{following}:

(i) $G$ is the direct union of $\{G_\mfm\}_{\mathfrak{m}\in\mathcal
M}$,

(ii) for $\mathfrak{m}\in\mathcal M$, $T_{\mathfrak{m}}$ is a
$G_\mfm$-torus, and $T$ is the direct union of $\{T_{\mathfrak
m}\}_{\mathfrak{m}\in\mathcal M}$,

(iii) $\supp(T)$ is the direct union of $\{\supp(T_{\mathfrak
m})\}_{{\mathfrak m}\in\mathcal M}$,

(iv) $\supp(T)$ is a PRS in $G$,

 (v) $T$ is domain, in particular, {it} has no nilpotents, and {it is
 strongly prime if $T$ is Jordan,}

 (vi) a nonzero element of $T$ is invertible if and only if is
 homogeneous,

(vii) if $0\not=x\in T$ and $x^m\in T^\sg$ for some $\sg\in G$,
$m\in\bbbz$, then $\sg\in mG$ and $x\in T^{\frac{1}{m}\sg}.$

\end{lem}

\proof The proof of parts (i)-(iii) is immediate. By {\cite[Lemma 3.5]{Yo1}},
for each ${\mathfrak m}\in\mathcal M$, $\supp(T_{\mathfrak m})$ is a
PRS in $G_\mfm$. Moreover, by (T1),
$G=\la\supp(T)\ra=\cup_{{\mathfrak m}\in\mathcal M}G_\mfm$. So by
part (iii) and Lemma \ref{b2}(ii), $\supp(S)$ is a PRS in $G$, proving (iv).
The proof of parts (v)-(vii) follows {from \cite[Lemma 3.6, Corollary
3.7 and Lemma 2.4]{Yo1}} and a direct union argument. \qed

We know from Lemma \ref{b3} that the center $Z(T)$ of $T$ is an
integral domain commutative associative homogeneous subalgebra of
$T$. In particular, $\Gamma:=\supp(Z(T))$ is a subgroup of $G$ and
$Z(T)$ is $\Gamma$-graded. It follows that $Z(T)$ is isomorphic to
{a commutative twisted group algebra}. The
group $\Gamma$ is called the {\it central grading group} of $T$. Let
$\bar{Z}$ be the field of fractions of $Z$, and consider
$\bar{T}=\bar{Z}\otimes_{Z}T$. If $T$ is an associative (Jordan)
algebra, then by Lemma \ref{b3}(v) and \cite[2.6]{Yo1}, $\bar{T}$ is
a domain associative (Jordan) algebra over $\bar{Z}$.

{The following lemma is proved in \cite[Lemma 3.9]{Yo1}, where $G$ is
assumed to be a free abelian group of finite rank. However, one can
check, using some straightforward modifications, that the same proof
works when $G$ is a torsion free abelian group.}

\begin{lem}\label{ipm6}
Let $G$ be a torsion free abelian group and $T=\bigoplus_{\a\in
G}T_\a$ be a Jordan or an associative torus. Let $Z=Z(T)$ be the
center of $T$ with the central grading group $\Gamma$. Let
$^{\overline{\;}}:G\longrightarrow G/\Gamma$ be the canonical map.
For $\a\in G$, let $T_{\bar \a}:=ZT_\a$ and
$\bar{T}_{\bar\a}:=\bar{Z}\otimes_Z ZT_\a$. Then

(i) $ZT_\a=ZT_\b$ for all $\a,\b\in G$ with $\a\equiv\b$ $\mod\;
\Gamma$,

(ii) $T=\bigoplus_{\bar\a\in G/\Gamma}T_{\bar\a}$ is a free
$Z$-module and a $G/\Gamma$-graded algebra over $Z$ with {$\hbox{rank}
T_{\bar\a}\leq 1$} for all $\bar\a\in G/\Gamma$,

(iii) $\bar{T}=\bigoplus_{\bar\a\in G/\Gamma}\bar{T}_{\bar\a}$ is a
$G/\Gamma$-graded torus over $\bar Z$ with $\dim_{\bar
Z}\bar{T}=|(\supp T)/\Gamma|$,

(iv) the quotient group $G/\Gamma$ {cannot} be a nontrivial cyclic
group.
\end{lem}

\subsection{\bf Associative
$G$-tori}\label{associative}\setcounter{equation}{0} Let $G$ be an
abelian group. Symbols $\sg,\tau,\mu$ always denote elements of $G$.
Let $\aa=\bigoplus_{\sg\in G}\aa^\sg$ be an associative $G$-torus. Since
homogeneous non-zero elements of $\aa$ are invertible, we have
$$\aa^\sg\aa^\tau=\aa^{\sg+\tau},\qquad(\sg,\tau\in\supp(\aa)).
$$
It follows that $\supp(\aa)$ is a subgroup of $G$ {and so by by (T1), $\supp(\aa)=G$.} For $\sg\in G$, we choose
$0\not=x^\sg\in \aa^\sg$. Then $\aa=\bigoplus_{\sg\in G}\bbbf x^\sg$.
Define $\lam:G\times G\longrightarrow \bbbf^\times$ by
\begin{equation}\label{gday1}
x^\sg x^\tau=\lam(\sg,\tau)x^{\sg+\tau},\qquad {(\sg,\tau\in G)}.
\end{equation}
Associativity of $\aa$ implies that $\lam$ is a $2$-cocycle, namely {for $\sg,\tau,\mu\in G$,}
\begin{equation}\label{4ipm}\lam(\sg+\tau,\mu)\lam(\sg,\tau)=\lam(\sg,\tau+\mu)\lam(\tau,\mu).
\end{equation}

Conversely, let $\lam:G\times G\longrightarrow\bbbf^\times$ be a
$2$-cocycle. Consider the abstract vector space $\aa:=\bigoplus_{\sg\in
G}\bbbf x^\sg$ with basis $\{x^\sg\mid\sg\in G\}$. Then the
multiplication on $\aa$ induced from (\ref{gday1}) makes $\aa$
into an associative $G$-torus with $\supp(\aa)=G$. We denote $\aa$
by $(\bbbf^t(G),\lam)$ and call it the {\it associative $G$-torus
determined by the $2$-cocycle $\lam$}. We note that the
associative $G$-torus $(\bbbf^t[G],\lam)$ can be characterized as
the  unital associative algebra  defined by {the set of} generators
$\{x^\sg\mid\sg\in g\}$ and relations \eqref{gday1}. The
associative $G$-torus $(\bbbf^t[G],\lam)$ is called {{\it elementary}} if
$\hbox{img}(\lam)\sub\{1,-1\}$. In the literature,
$\aa=(\bbbf^t[G],\lam)$ is also known as the {\it twisted group
algebra} determined by $\lam$ (see \cite{OP}). We summarize the
above discussion as follows.

\begin{lem}\label{tori-11}
Let $G$ be an abelian group and $\aa$ be an associative algebra.
Then $\aa$ is a $G$-torus if and only if
$\aa\cong_G(\bbbf^t[G],\lam)$ for a $2$-cocycle $\lam$.
\end{lem}

Let $\aa=(\bbbf^t[G),\lam)$ be an associative $G$-torus. Note that
for $\sg,\tau\in G$, $x^\sg$ and $x^\tau$ commute up to a
''twisting'', namely
\begin{equation}\label{tat1}
{x^\sg x^\tau=\lam_t(\sg,\tau)x^\tau x^\sg,}
\end{equation} where
\begin{equation}\label{quan}
\lam_t(\sg,\tau):=\lam(\sg,\tau)\lam(\tau,\sg)^{-1}.
\end{equation}
We clearly have $\lam_t(\sg,\sg)=1$ and
$\lam_t(\sg,\tau)=\lam_t(\tau,\sg)^{-1}$. Moreover, one can check
that $\lam_t:G\times G\longrightarrow\bbbf^\times$ is a group
bihomomorphism.

\begin{rem}\label{gday+12}
{Suppose in the above discussion that we replace the basis
$\{x^\sg\mid\sg\in G\}$ of $\aa$ by another basis $\{y^\sg\mid\sg\in
G\}$. Then for $\sg\in G$, $y^\sg=d(\sg)x^\sg$ where
$d:G\longrightarrow \bbbf^\times$ is a map. Denote the corresponding
$2$-cocycle as in (\ref{gday1}) by $\hat\lam:G\times
G\longrightarrow\bbbf^\times$. Then we have
$$\hat\lam(\sg,\tau)=d(\sg)d(\tau)d(\sg+\tau)^{-1}\lam(\sg,\tau).$$
Therefore, $\lam$ and $\hat\lam$ are {\it equivalent}, up to a
coboundary. That is the product on $\aa$ is uniquely determined
up to $H^2(G,\bbbf)$ (see \cite[\S 1]{OP}).}
%
\end{rem}

\begin{exa}\label{gday-1}{\em  {\bf (Quantum tori).} Let $\Lam$ be a free
abelian group of rank $|I|$ where $I$ is a nonempty index set with a fixed
total ordering $<$. Let $\aa=(\bbbf^t[\Lam],\lam)$ be a
$\Lam$-torus determined by a $2$-cocycle $\lam$. We fix a basis
$\{\sg_i\mid i\in I\}$ of $\Lam$, and set
$q_{ij}=\lam_t(\sg_i,\sg_j)$. Since $\lam_t$ is a bihomomorphism,
we have, for $\sg=\sum_{i\in I}n_i\sg_i$ and $\tau=\sum_{i\in
I}m_i\sg_i$, $\lam_t(\sg,\tau)={\prod_{i,j}q_{ij}^{n_im_j}}$, with
$q_{ij}=q_{ji}^{-1}$ and $q_{ii}=1$ for all $i,j\in I$. We note
that as $n_i$'s and $m_i$'s are zero almost for all $i$, the above
product makes sense. In the literature a matrix $(q_{ij})_{i,j\in
I}$ (possibly of infinite rank)  satisfying $q_{ii}=1$ and
$q_{ij}=q_{ji}^{-1}$, for all $i,j$, is called a {\it quantum
matrix}. A quantum matrix $\bq$ is called {\it elementary} if
$q_{ij}\in\{\pm 1\}$ for all $i,j$. For $i\in I$, we set
$y_i:=x^{\sg_i}$. Also for $\sg=\sum_{i\in I}n_i\sg_i$, we set
{$y^\sg=1$ if $\sg=0$ and if $\sg\not=0$, we set} $y^\sg:=y_{i_1}^{n_{i_1}}\cdots y_{i_k}^{n_{i_k}}$ where
$i_1<\cdots <i_k$ are all indices for which $n_{ij}\not=0$. Then
we have $\aa=\bigoplus_{\sg\in\Lam}\bbbf y^\sg$, and for all
$i,j\in I$,
\begin{equation}\label{eq1}
y_iy_j=q_{ij}y_jy_i\andd y_iy_i^{-1}=y_i^{-1}y_i=1.
\end{equation}
{The $\Lam$-torus} $\aa$ can be described as the unital
associative algebra defined by generators {$1, y_i,y_i^{-1}$} and
relations (\ref{eq1}), induced from the quantum matrix
$\bq:=(q_{ij})$. In this case, we denote $\aa$ by
$\aa=(\bbbf^t[\Lam],\bq)$ and call it the {\it quantum torus}
determined by the quantum matrix $\bq$.}
\qedexa
\end{exa}

Here is a generalization of \cite[Lemma 4.6]{Yo1} to the torsion free
case.

\begin{lem}\label{tori-12}
Let $G$ be a torsion free abelian group and $\aa$ be an associative
algebra. If $\aa^+$ is a Jordan $G$-torus, then
$\aa\cong_G(\bbbf^t[G],\lam)$ for some $2$-cocycle $\lam$. In
particular, if $G$ is free abelian, then $\aa\cong_G(\bbbf^t[G],\bq)$
for some quantum matrix $\bq$.
\end{lem}

\proof By Lemma \ref{tori-11}, we must show that $\aa$ is an
associative $G$-torus. Since $\aa^+$ is a Jordan torus, we have
$\aa^+=\aa=\bigoplus_{\sg\in g}\aa^\sg$ with $G=\la\sg\in G\mid
\aa^\sg\not=0\ra$ and $\dim\aa^\sg\leq 1$ for all $\sg\in G$. So it
only remains to show that $\aa$ is $G$-graded, namely
$\aa^\sg\aa^\tau\sub\aa^{\sg+\tau}$, for all $\sg,\tau\in G$. We
{proceed with showing} this for fixed $\sg,\tau\in G$. We may assume without
loos of generality that both $\aa^\sg$ and $\aa^\tau$ are non-zero.
Let $0\not=x\in\aa^\sg$ and $0\not=y\in\aa^\tau$. {By Lemma
\ref{b3}, $x$ and $y$ are invertible in $\aa^+$ and so they are invertible in $\aa$. Therefore $xy$
and $yx$ are invertible in $\aa$ and so in $\aa^+$. Then by Lemma \ref{b3}, both $xy$ and $yx$ are
homogenous in $\aa$.} Now as $x\circ y=xy+yx\in\aa^{\sg+\tau}$, we
conclude that $xy\in\aa^{\sg+\tau}$ if $x\circ y\not=0$. Suppose now
that $x\circ y=0$. Then $xy=-yx\in\aa^\d$ for some $\d\in G$ and as
$\aa$ is associative, $(xy)^2=-x^2y^2=-y^2x^2$. Therefore,
$$0\not=(xy)^2=-\frac{1}{2}(x^2y^2+y^2x^2)=-\frac{1}{2}(x^2\circ
y^2)\in\aa^{2\d}\cap\aa^{2\sg+2\tau}.$$ Thus $\d=\sg+\tau$. The
second statement follows immediately from Example \ref{gday-1}.\qed

\subsection{\bf Involutorial associative
$G$-tori}\label{associative-involution}\setcounter{equation}{0}

Let
$\aa=(\bbbf^t[G],\lam)$ be an associative $G$-torus. Assume further
that $\aa$ is equipped with a graded involution $\bar{\;}$. Namely a
period $2$ anti-automorphism $\bar{\;}$ satisfying
$\overline{\aa^\sg}=\aa^\sg$, $\sg\in G$. Then for $\sg\in G$, we
have $\overline{x^\sg}=a_\sg x^\sg$ where $a_\sg\in\bbbf^\times$
satisfies $a_\sg^2=1$. So, for $\sg\in G$,
$$\overline{x^\sg}=(-1)^{q(\sg)}x^\sg$$
where $q$ is map from $G$ into the field $\bbbf_2$ of $2$
elements. We note that
$$(-1)^{q(\sg+\tau)}x^{\sg+\tau}=\overline{x^{\sg+\tau}}=\lam(\sg,\tau)^{-1}\overline{x^\tau}\;\overline{x^\sg}=
\lam(\sg,\tau)^{-1}\lam(\tau,\sg)(-1)^{q(\sg)+q(\tau)}x^{\sg+\tau}.$$
Thus
\begin{equation}\label{gday5}
(-1)^{\b_q(\sg,\tau)}=\lam_t(\sg,\tau),
\end{equation}
where $\b_q:G\times G\longrightarrow\bbbf_2$ is defined by
$\b_q(\sg,\tau)=q(\sg)+q(\tau)-q(\sg+\tau)$. Now $\lam_t$ being a
bihomomorphism implies that $\b_q$ is also a group bihomomorphism.
Therefore, by definition, $q:G\longrightarrow\bbbf_2$ is a quadratic
map.

Conversely, starting from an associative $G$-torus
$\aa=(\bbbf^t[G],\lam)$ and a quadratic map $q:G\longrightarrow
\bbbf_2$ satisfying $(-1)^{\b_q(\sg,\tau)}=\lam_t(\sg,\tau), $ one
can {define} a graded involution $\bar{\;}$ on $\aa$ by
$\overline{x^\sg}=(-1)^{q(\sg)}x^\sg$. In fact, it is clear that
$\bar{\;}$ is a period $2$ isomorphism of $\bbbf$-vector spaces.
Moreover, for $\sg,\tau\in G$, we have
\begin{eqnarray*} \overline{x^\sg x^{\tau}}&=&\lam(\sg,\tau)\overline{x^{\sg+\tau}}\\
&=&\lam(\sg,\tau)(-1)^{q(\sg+\tau)}x^{\sg+\tau}\\
&=&\lam(\sg,\tau)(-1)^{q(\sg+\tau)}\lam(\tau,\sg)^{-1}x^\tau x^\sg\\
&=&(-1)^{\beta_q(\sg,\tau)+q(\sg,\tau)} x^\tau x^\sg\\
&=&(-1)^{q(\sg)+q(\tau)}x^\tau x^\sg\\
&=& \overline{x^\tau}\;\overline{x^\sg}.
\end{eqnarray*}

\begin{DEF}\label{gday4}{\em
Let $\aa=(\bbbf^t[G],\lam)$ be a $G$-torus and
$q:G\longrightarrow\bbbf_2$ be a quadratic map satisfying (\ref{gday5}).
We denote the induced involution on $\aa$ by $\theta_q$. We recall that in this case
$H(\aa,\theta_q)=H((\bbbf^t[G],\lam),\theta_q)$ is a subalgebra of $\aa^+$.}
\end{DEF}

{Let $J$ be a Jordan $G$-torus. By Lemma \ref{b3}, $J$ is strongly prime, so by Zelmanov's Prime Structure Theorem \cite[pg. 200]{MZ},
$J$ has one of the types, {\it Hermitian}, {\it Clifford}, or {\it Albert}. We recall that
$J$ is of Hermitian type if $J$ is special and $q_{48}(J )\not=\{0\}$ (the term
$q_{48}(J)$ will be explained in the next section). Also $J$ is of Clifford type if the central closure
$\bar{J}$ is a Jordan algebra over $\bar{Z}$ of a
symmetric bilinear form.
Finally $J$ is of Albert type if the central closure $\bar{J}$ is an Albert algebra over $\bar{Z}$. In the remaining sections,
we study each of the mentioned types separately.}

\renewcommand{\theequation}{\thesection.\arabic{equation}}
\section{\bf {Jordan tori of Hermitian type}}\label{Hermitian-type}\setcounter{equation}{0}
{Throughout this section $G$ is a torsion free abelian group, unless otherwise mentioned. All associative algebras are assumed to be unital.
We assume that any algebra homomorphism from  a unital
algebra to a unital algebra maps 1 to 1. We recall that a Jordan torus $J$ is called a {\it Hermitian torus}
if there exists an involutorial associative algebra $(\aa,*)$ which is $*$-prime such that $\aa$ is generated by $J$ and $J = H(\aa, *)$.
}

We make a convention that for two elements  $x,y$ of  an associative algebra, by $[x,y],$ we mean $xy-yx$ and by $x\circ y,$ we mean $xy+yx.$
Suppose that $X$ is an infinite set and $\mathfrak{a}(X)$ is the free
associative algebra on $X.$ We consider the special Jordan algebra
$\mathfrak{a}(X)^+$ and take $\mathfrak{fsj}(X)$ to be the subalgebra
of $\mathfrak{a}(X)^+$ generated by $X.$ We refer to
$\mathfrak{fsj}(X)$ as  the {\it free special Jordan algebra on $X.$}
We recall that an ideal $I$ of $\fsj(X)$ is called {\it formal} if
for each polynomial $p(x_1,\ldots,x_n)\in I$ with $x_1,\ldots,x_n\in
X,$ and each permutation $\sg$ of $X,$ one concludes
$p(\sg(x_1),\ldots,\sg(x_n))\in I.$
A formal ideal $H$ of $\fsj(X)$
is called {\it Hermitian} if it is closed under $n$-tads
($n\in\bbbn^{\geq 4}$), i.e., for $x_1,\ldots,x_n\in H,$ ($n\in\bbbn^{\geq 4}$),
$\{x_1,\ldots,x_n\}:=x_1\cdots x_n+x_n\cdots x_1\in H.$
Now suppose that $H(X)$ is a Hermitian ideal of
$\fsj(X).$ For  an $i$-special Jordan algebra (i.e., a quotient algebra of a special Jordan algebra)  $J,$
by $H(J),$ we mean the evaluation of $H(X)$ on $J;$ $H(J)$ is called
a {\it Hermitian part} of $J.$ It is well known that a Hermitian part
$H(J)$ of an $i$-special Jordan algebra $J$ corresponding to a
Hermitian ideal $H(X)$ of $\fsj(X)$ is an ideal of $J.$

Now for $x,y,z,w\in X,$ we take $D_{x,y}(z):=[[x,y],z]$ and set $$p_{16}(x,y,z,w):=[[D^2_{x,y}(z)^2,D_{x,y}(w)],D_{x,y}].$$ Then $q_{48}:=[[p_{16}(x_1,y_1,z_1,w_1),p_{16}(x_2,y_2,z_2,w_2)],p_{16}(x_3,y_3,z_3,w_3)]$ is a polynomial in the free associative algebra on $X$
in 12 variables $x_i,y_i,z_i,w_i,$ $1\leq i\leq 4.$ Take  $Q_{48}$ to
be the  linearization-invariant $T$-ideal of $\fsj(X)$ generated by
$q_{48}.$ It means that $Q_{48}$ is the smallest ideal of
$\fsj(X)$ containing $q_{48}$ with the following two properties, if
$p$ is a polynomial in $Q_{48},$ then each linearization of $p$ is
also an element of $Q_{48}$ and that $Q_{48}$ is invariant
under all algebra endomorphisms of $\fsj(X).$ We note that  for 12
variables $x_i,y_i,z_i,w_i,$ $1\leq i\leq 4,$ each monomial of $q_{48}$
is a product of 12 variables $x_i,y_i,z_i,w_i,$ $1\leq i\leq 4,$ and
monomials have the same number of $x\in\{x_i,y_i,z_i,w_i\mid 1\leq
i\leq 4\}.$ So each polynomial  in $Q_{48}$ is a summation of
monomials having the same partial degree. using the same argument as in \cite[Lemma 4.1]{Yo1}, we have the following lemma:
\begin{lem}
\label{q48} Suppose $J$ is a Jordan
$G$-torus of Hermitian type, then $J=H(P,*)$ for an associative algebra {$P$} with an involution $*$ such that $P$ is
$*$-prime and {is} generated by $J.$
\end{lem}

\begin{lem}
\label{4.9 yoshii} {Suppose that $J$ is a Jordan $G$-torus.} Suppose that $\bbbe$ is a quadratic field
extension of $\bbbf,$ $\sg_\bbbe$ is the nontrivial Galois automorphism and $\lam:G\times G\longrightarrow \bbbe^\times$
is a $2$-cocycle. Assume $\bbbe\otimes_\bbbf J\simeq_G(\bbbe^t[G],\lam)^+,$ say via $\varphi,$
then either there is a $2$-cocycle $\mu:G\times G\longrightarrow
\bbbf^\times$ such that $(\bbbe^t[G],\lam)\simeq_G (\bbbe^t[G],\mu)$
and $J\simeq_G (\bbbf^t[G],\mu)^+,$ or $J\simeq_G
H((\bbbe^t[G],\mu),\theta)$ for some $2$-cocycle $\mu$ satisfying $\sg_\bbbe(\mu(g_1,g_2))=\mu(g_2,g_1)$ for $g_1,g_2\in G,$ and an
$\sg_\bbbe$-semilinear anti-automorphism  $\theta,$ where  $(\bbbe^t[G],\mu)$ is considered as an
$\bbbf$-algebra.
\end{lem}

\begin{proof}
Since   $\bbbe$ is a quadratic field  extension of
$\bbbf,$ there is an  irreducible polynomial on $\bbbf$ of degree 2 with distinct roots $e,f.$ Then
$\bbbe=\bbbf+e\bbbf$ {and} $\sg_\bbbe:\bbbe\longrightarrow \bbbe$ is
 the Galois automorphism  mapping   $e$ to $f.$ Set
 $\tau:=\sg_\bbbe\otimes id:\bbbe\otimes J\longrightarrow \bbbe\otimes J.$ Since for $x,y\in E$ and $a\in J,$ we have $\tau(xy\otimes a)=\sg_\bbbe(xy)\otimes a=\sg_\bbbe(x)\sg_\bbbe(y)\otimes a=\sg_\bbbe(x)(\sg_\bbbe(y)\otimes a),$ we get that $\tau$ is a $\sg_\bbbe$-semilinear automorphism  of
the Jordan algebra $\bbbe\otimes J.$ Consider the $\bbbe$-Jordan algebra  isomorphism  $\varphi:\bbbe\otimes_\bbbf J\longrightarrow (\bbbe^t[G],\lam)^+,$
 then $\theta:=\varphi\tau\varphi^{-1}$ is a Jordan $\sg_\bbbe$-semilinear automorphism on $(\bbbe^t[G],\lam)^+.$
 Next we note that as $\theta:=\varphi\tau\varphi^{-1}$ is a Jordan $\sg_\bbbe$-semilinear automorphism  on $(\bbbe^t[G],\lam)^+,$  it is also  an $\bbbf$-linear automorphism of the $\bbbf$-Jordan algebra  $(\bbbe^t[G],\lam)^+.$ So by \cite[Lemma 1.1.7]{J1} either $\theta$ is an associative algebra $\sg_\bbbe$-semilinear automorphism on $(\bbbe^t[G],\lam)$ or it is an associative algebra  $\sg_\bbbe$-semilinear anti-automorphism  on $(\bbbe^t[G],\lam)$ which is not an automorphism.
 We know that  $\bbbf\otimes J=H(\bbbe\otimes J,\tau)$ and  the restriction of $\varphi$ to $H(\bbbe\otimes J,\tau)$ is an $\bbbf$-Jordan algebra isomorphism from $J\simeq H(\bbbe\otimes J,\tau)$ to $H((\bbbe^t[G],\lam),\theta).$ So to complete the proof, we show that either $H((\bbbe^t[G],\lam),\theta)= (\bbbf^t[G],\mu)^+,$ for a 2-cocycle $\mu:G\times G\longrightarrow\bbbf$ or $H((\bbbe^t[G],\lam),\theta)=
H((\bbbe^t[G],\mu),\theta)$ for some $2$-cocycle $\mu$ satisfying $\sg_\bbbe(\mu(g_1,g_2))=\mu(g_2,g_1)$ for $g_1,g_2\in G.$

We fix $0\neq x^g\in J_g$ ($g\in G$), so $\{x^g\mid g\in G\}$ is an $\bbbf$-basis for $J$ and an $\bbbe$-basis for $\bbbe\otimes_\bbbf J$ (here  we identify $J$ with  $\{1\otimes x\mid x\in J\}\sub\bbbe\otimes J$). Now as for $g\in G,$ $\tau(x^g)=x^g$ and $\varphi$ is a $G$-graded isomorphism, $\{y^g:=\varphi(x^g)\mid g\in G\}$ is  a basis for the $\bbbe$-vector space $\bbbe^t[G]$ consisting of  homogeneous elements fixed by $\theta.$ Now let $\mu:G\times G\longrightarrow \bbbe$ be the $2$-cocycle corresponding to this new basis; see Remark \ref{gday+12}. We note that
 $$(\bbbe^t[G],\lam)=(\bigoplus_{g\in G}\bbbe y^\sg,\mu)=\bigoplus_{g\in G}\bbbf y^\sg+\bigoplus_{g\in G}e\bbbf y^\sg$$ and $$H((\bbbe^t[G],\lam),\theta)=\bigoplus_{g\in G}\bbbf y^\sg.$$

Now we consider the  two cases that either $\theta$ is an associative algebra $\sg_\bbbe$-semilinear automorphism on $(\bbbe^t[G],\lam)=(\bbbe^t[G],\mu)$ or it is an associative algebra  $\sg_\bbbe$-semilinear anti-automorphism  on $(\bbbe^t[G],\lam)=(\bbbe^t[G],\mu)$ which is not an automorphism. In the former case, for $g_1,g_2\in G,$
    we have
    $0=\theta(y^{g_1}y^{g_2}-\mu(g_1,g_2)y^{g_1+g_2})=y^{g_1}y^{g_2}-\sg_\bbbe(\mu(g_1,g_2))y^{g_1+g_2}.$
    So we have
    $\mu(g_1,g_2)y^{g_1+g_2}=y^{g_1}y^{g_2}=\sg_\bbbe(\mu(g_1,g_2))y^{g_1+g_2}.$
    Therefore $\mu(g_1,g_2)\in\bbbf.$ Now  we have
    $$(\bbbe^t[G],\lam)=\bigoplus_{g\in G}\bbbe y^\sg=\bigoplus_{g\in
    G}\bbbf y^\sg+\bigoplus_{g\in G}e\bbbf y^\sg$$ and  as $\mu(G,G)\sub\bbbf,$  $\bigoplus_{g\in G}\bbbf
    y^\sg$ is closed under the associative product on
    $(\bbbe^t[G],\mu)$ and $H((\bbbe^t[G],\lam)=
    H((\bbbe^t[G],\mu),\theta)=\bigoplus_{g\in G}\bbbf y^\sg$
    can be identified with $(\bbbf^t[G],\mu)^+.$ In the latter case, for $g_1,g_2\in G,$
    we have
    $0=\theta(y^{g_1}y^{g_2}-\mu(g_1,g_2)y^{g_1+g_2})=y^{g_2}y^{g_1}-\sg_\bbbe(\mu(g_1,g_2))y^{g_1+g_2}.$ So  $\sg_\bbbe(\mu(g_1,g_2))=\mu(g_2,g_1).$
This completes the proof.
\end{proof}

The following generalizes \cite[Proposition 4.7]{Yo1} to
the torsion free case.

\begin{pro}\label{jordan-14}
{Let $\aa$ be an involutorial associative
algebra and assume that $J:=H(\aa,*)$ is a  Jordan $G$-torus generating
$\aa$.}

(a) Suppose that there exists $a\in\aa$ such that $aa^*=0$ and
$a+a^*$ is invertible in $J$. Then $J\cong_G(\bbbf^t[G],\lam)^+$ for
some $2$-cocycle $\lam$.

(b) Suppose there exists an invertible element $a\in\aa$ such that
$a^*=-a$ and $0\not=y\in J_\gamma$ for some $\gamma\in G$ such that
$a^2\in J_{2\gamma}$, $ay^{-1}a\in J_\gamma$ and $[a,y]\in
J_{2\gamma}$. Then $J\cong_G(\bbbf^t[G],\lam)^+$ or $E\otimes_\bbbf
J\cong_G(\bbbe^t[G],\lam)^+$ for some $2$-cocycle $\lam$.
\end{pro}

\proof (a) By Lemma \ref{b3}(v), $J$ is domain. By \cite[Lemma
4.5]{Yo1}, $J\cong\aa^+$ for some associative algebra $\aa$.
Then by Lemma \ref{tori-12}, we are done. The proof of part (b) is
exactly the same as \cite[Proposition 4.7(b)]{Yo1}. \qed

\begin{DEF}\label{nolable}
{\rm
{A Jordan $G$-torus $J$  is
said to be of {\it involution type} if we have $J\cong_G
H((\bbbf^t[G],\lam),\theta_q),$ ($\lam$ a $2$-cocycle and $q$ a quadratic
map), it is said to be of {\it plus type} if
$J\cong_G(\bbbf^t[G],\lam)^+$ ($\lam$ a $2$-cocycle), finally it is
said to be of {\it extension type} if $J\cong_G
H((\bbbe^t[G],\lam),\sg),$ ($\bbbe$ a quadratic field extension of
$\bbbf,$ $\lam$ a $2$-cocycle and $\sg$ an involution). {If $G$ is free abelian of finite rank, we call
a Jordan $G$-torus of one of the above types, simply a Jordan torus of that type.}}}
\end{DEF}

\begin{lem}
\label{pre}
Suppose that $G$ is  a free abelian group of finite rank and $J=\op_{g\in G}J^g$ is a Jordan $G$-torus. Suppose that $P$ is an associative algebra with involution  $*$ and $J=H(P,*).$ For $g\in \supp(J),$ fix $0\neq x_g\in J^g.$ If for all $g,h\in \supp(J),$ $x^gx^h=\pm x^hx^g$ (product in $P$), then
one of the following occurs:

(a) $P$ is  isomorphic to $(\bbbf^t[G],\lam)$ for a 2-cocycle $\lam:G\times G\longrightarrow \bbbf;$ in particular, $P$ is a $G$-graded algebra. Moreover $*$ is a $G$-graded involution  and $J$ is graded isomorphic to $H((\bbbf^t[G],\lam),\theta_q),$ where   $q:G\times G\longrightarrow \bbbf_2$ is the quadratic map  arising from the involution on $(\bbbf^t[G],\lam)$ induced via the isomorphism form $P$ to $(\bbbf^t[G],\lam)$ (see Section \ref{associative-involution}); in particular, $P^g=J^g$ for all $g\in \supp(J).$

{(b) There are an invertible element $u$ of $ P$ and a nonzero element  $y$ of $ J^\gamma$ for some $\gamma\in G$ such that the following four conditions hold:
$$u^*=-u,\;u^2\in J^{2\gamma},\;uy^{-1}u\in J^\gamma,\;[u,y]\in J^{2\gamma}.$$}
\end{lem}
\begin{proof}
One knows from Lemma \ref{b3} and {\cite[Proposition II.1.11]{AABGP}} that there is a basis $B=\{\sg_1,\ldots,\sg_n\}\sub\hbox{supp}(J)$ for $G.$ If  $J$ is generated by $r$-tads $\{x_{\sg_{i_1}}^{\ep_{1}}\cdots x_{\sg_{i_r}}^{\ep_{r}}\}$ for $r\in\bbbz^{>0},$ $1\leq i_1,\ldots,i_r\leq n,$ $\ep_1,\ldots,\ep_r\in\{\pm1\},$ the conditions (A) and (B) of the proof \cite[Thm. 4.11]{Yo1} hold and so by the proof of the same theorem all the statements in (a) are fulfilled.  But if   $J$ is not generated by $r$-tads $\{x_{\sg_{i_1}}^{\ep_{1}}\cdots x_{\sg_{i_r}}^{\ep_{r}}\}$ for $r\in\bbbz^{>0},$ $1\leq i_1,\ldots,i_r\leq n,$ $\ep_1,\ldots,\ep_r\in\{\pm1\},$ the condition (A) but not  (B) of the proof \cite[Thm. 4.11]{Yo1} holds and again by the proof of the same theorem there are  an invertible element $u$ of $ P$ and a nonzero element  $y$ of {$ J^\gamma$} for some $\gamma\in G$  such that $u^*=-u,\;u^2\in J^{2\gamma},\;uy^{-1}u\in J^\gamma,\;[u,y]\in J^{2\gamma}.$
\end{proof}

\begin{pro}\label{classification}
Suppose that $P$ is an associative algebra with an involution $*.$ Suppose that  $J:=H(P,*)$ generates $P.$ Suppose
that $\{G_i\mid i\in\i\}$ is a class of free abelian
subgroups of $G$ such that $G=\cup_{i\in \i} G_i.$ Also assume $J=\op_{g\in G}J^g$ is a Jordan $G$-torus. Set $$J_i:=\op_{g\in G_i}
J^g;\;\;(i\in\i).$$ Assume that $ *_i,$ the restriction of $*$ to the
subalgebra $P_i$ ($i\in\i$) of $P$ generated by $J_i$ is an
involution of $P_i$ and that $J_i=H(P_i,*).$ If $P$ is $*$-prime,
then one of the following holds for  $J$:
\begin{itemize}
\item{$J\simeq H((\bbbf^t[G],\lam),\theta_q)$, $\lam$ a $2$-cocycle and $q$ a quadratic map}
\item{$J\simeq(\bbbf^t[G],\lam)^+$, $\lam$ a $2$-cocycle}
\item{$J\simeq H((\bbbe^t[G],\lam),\sg)$, $\bbbe$ a quadratic field extension of $\bbbf,$ $\lam$ a $2$-cocycle and $\sg$ an involution.}
\end{itemize}
Moreover, if $J$ is of involution (resp. plus or extension) type,
it is a direct union of Jordan tori of involution (resp. plus or
extension) type.
\end{pro}
\begin{proof}
We know that  $J=\op_{g\in G}J^g$ is a Jordan $G$-torus. For $g\in \hbox{supp}(J),$ we fix $0\neq x_g\in J^g.$
 We consider the following two cases:

\underline{\textbf{Case 1}: For all $g,h\in \supp(J),$ $x_gx_h=\pm x_hx_g.$}
 By Lemma \ref{pre}, one of the following occurs:

(a) for all $i\in \i,$
 \begin{itemize}
\item  $P_i$ is equipped with a $G_i$-grading $\bigoplus_{g\in G}P_i^g$ with $P_i^g=J_i^g,$ for all  $g\in \supp(J_i),$
\item $P_i=\bigoplus_{g\in G_i}P_i^g$ is an associative $G_i$-torus,
\item $*_i=*\mid_{P_i}$ is a $G_i$-graded involution,
\end{itemize}

(b) there is $i\in\i$ for which there are an invertible element $u$ of $ P_i$ and a nonzero element  $y$ of $ J_i^\gamma$ for some $\gamma\in G_i$ such that the following four conditions hold:
$$u^*=-u,\;u^2\in J^{2\gamma},\;uy^{-1}u\in J^\gamma,\;[u,y]\in J^{2\gamma}.$$

We now assume (a) is satisfied, $i,j\in \i$ with $i\preccurlyeq j$ and $g\in G_i.$ If $g\in \supp(J_i),$ then $P_i^g=J_i^g=J_j^g=P_j^g.$ Also we know that $P_i$ is generated by $J_i$ and so $P_i$ is generated by $\cup_{g\in G_i} J_i^g=\cup_{g\in \supp(J_i)}J_i^g,$ in particular, for $g\in G_i,$ there are $\tau_1,\ldots,\tau_t\in \supp(J_i)$ such that $g=\tau_1+\cdots+\tau_t$ and
\begin{eqnarray*}
P_i^g=J_i^{\tau_1}\cdots J_i^{\tau_t}&=&P_i^{\tau_1}\cdots P_i^{\tau_t}\\
&=&P_j^{\tau_1}\cdots P_j^{\tau_t}\sub P_j^g.
\end{eqnarray*}
Therefor we have proved
$$P_i^g=P_j^g;\;\;\; i,j\in\i,\;i\preccurlyeq j,\;    g\in G_i.$$
So by Lemma \ref{direct-limit}, $P$ is an associative  $G$-torus with $P^g=J^g$ for all $g\in \supp(J).$
{Therefore} $J$ is  graded isomorphic {to} $H((\bbbf^t[G],\lam), \theta_q)$ {for a
$2$-cocycle $\lam:G\times G\longrightarrow\bbbf^\times$ and a quadratic map $q:G\times G\longrightarrow \bbbf_2.$}

Next we assume (b) is satisfied. Then we are done by Lemma \ref{jordan-14}.

\underline{\textbf{Case 2}: There are $g,h\in \supp(J)$ such that  $x_gx_h\neq\pm x_hx_g.$}
Set
$u:=[x_g,x_h]\neq 0$ and $d:=x_g\circ x_h\neq 0.$ We have one of the following conditions:

\begin{itemize}
\item $u^2=0:$ We have $u=-u^*$ and so $uu^*=0,$ then there
    exists $y\in J$ such that for $v:=yu,$ $v+v^*\neq 0.$
    Otherwise, for all $y\in J,$ we have $v+v^*=0$ in which
    $v:=yu.$ So we have $yu=v=-v^*=-u^*y^*=uy.$ Therefore for $w\in P,$ we have
    $(uy)(uw)=u(yu)w=u(uy)w=u^2w=0.$ This implies that
    $(uJ)(uP)=\{0\}.$ Now as $J$ generates $P,$ we get
    $(uP)^2=\{0\}.$ So we have $(PuP)^2=PuPPuP\sub
    PuPuP=P(uP)^2=\{0\}.$ Also $(PuP)^*=PuP$ and so $PuP$ is a
    nonzero $*$-ideal of $P$ with $(PuP)^2=\{0\},$ a contradiction, as $P$ is $*$-prime. Therefore
    there exists $y\in J$ such that for $v:=yu,$ $v+v^*\neq 0.$
    Since $y\in J,$ we have $y=\sum_{g\in G}y_g$ $(y_g\in J_g).$
    For $g\in G,$ set $v_g:=y_gu.$ If for all  $g\in G,$
    $v_g+v_g^*=0,$ we get $$v+v^*=[y,u]=\bigoplus_{g\in
    G}[y_g,u]=\bigoplus_{g\in G}(y_gu-uy_g)=\bigoplus_{g\in G}v_g+v_g^*=0$$ which is a
    contradiction. So there is $g_*\in G$ with
    $v_{g_*}+v_{g_*}^*\neq0.$ Now as $v_{g_*}+v_{g_*}^*$ is a
    homogeneous element of $J$ (see \cite[(2.7)]{Yo1}), it is invertible. Also
    $v_{g_*}v_{g_*}^*=0.$ So setting $c:=v_{g_*}\in  P,$ we get
    that $cc^*=0$ and $c+c^*$ is invertible in $J.$ Now fix
    $r_*\in\i$ with $y\in J_{r_*}$ and $u\in P_{r_*}.$ Now  for all $i\in \i$ with
    $r_*\preccurlyeq i,$ we have $c\in J_{r_*},$ $cc^*=0$ and $c+c^*$ is
    invertible in $J_{r_*}.$  We know that $P=\cup_{i\in
    \i_*}P_i$ in which $\i_*=\{i\in\i\mid r_*\preccurlyeq i\}$
    and $J=\cup_{ i\in\i_*}J_i.$ So $J$ is the direct union of
    Jordan tori $J_i$'s each of which contains the  element $c.$
    Since $c+c^*$ is invertible in $J_{r_*},$ this is invertible
    in each $J_i$ ($i\in\i_*$).  Now as $cc^*=0,$ we get that $J$
    as well as  each $J_i,$ $i\in\i,$ is of {plus type} by
    Proposition \ref{jordan-14}.


\item  $u^2\neq 0.$ We have  $d=x_g\circ x_h \in J^{g+h}.$ So there is $j_*\in \i$ with $d\in J_{i_*}^{g+h}.$ Now we have

$u^2\in(J_{i_*})_{2\gamma}$

$ud^{-1}u\in(J_{i_*})_{\gamma}$

$[u,d]\in (J_{i_*})_{2\gamma}$
(see \cite[Page 24]{Yo1}).

\noindent Then $P$ is the direct union of $P_i$'s ($i\in \i_*$)      where
$\i_*:=\{i\in\i\mid i_*\preccurlyeq i\}$        and $J$ is the
direct union of $J_i$'s for $i\in\i_*.$   Now using the proof of
\cite[Pro. 4.7]{Yo1} together with \cite[Pro. 4.9]{Yo1} either
each $J_i$  ($i\in\i_*$) is of {plus type}  or each $J_i$
($i\in\i_*$) is of {extension type}. Moreover  by Lemmas
\ref{jordan-14} and \ref{4.9 yoshii}, either $J$ {is} of {plus type}  or
of {extension type} respectively.
\end{itemize}
\end{proof}

 \begin{thm}\label{92-1}
 {Suppose that $J$ is a Jordan $G$-torus of Hermitian type, then  $J$ is a direct union of Jordan tori of Hermitian type and it is of one of involution, plus or extension types. Moreover if $J$ is of involution (resp. plus or extension) type, it is a direct union of Jordan tori of involution (resp. plus or extension) type.}
 \end{thm}

\begin{proof}
The group $G$ is a torsion free abelian group and  $J=\op_{g\in
G}J^g$ is a Jordan  $G$-torus of Hermitian type. Take $S$ to be the
support of $J.$ Since $J$ is of Hermitian type, $q_{48}(J)\neq0.$ Fix
$x_1,\ldots,x_{12}\in J$ such that $q_{48}(x_1,\ldots,x_{12})\neq0.$
Since $J=\op_{\sg\in G}J^\sg,$ there are $\sg_1,\ldots,\sg_n\in G$
such that $x_1,\ldots,x_{12}\in J^{\sg_1}\op\cdots\op J^{\sg_n}.$ Now
suppose $\mathcal{I}:=\{T\sub S\mid \sg_1,\ldots,\sg_n\in T,
|T|<\infty\}.$ Set $$G_{_T}:=\la T\ra\andd S_{_T}:=S\cap G_{_T}\;\;
(T\in\mathcal{I}).$$ Then $S=\cup_{T\in\mathcal{I}} S_{_T}$ and
$G_{_T}=\la S_{_T}\ra.$ Next set $$J_{_T}:=\op_{\sg\in G_{_T}}J^\sg
\;\;(T\in \mathcal{I}).$$ One has $J=\cup_{T\in\mathcal{I}} J_{_T}$
and that each $J_{_T}$ is a Jordan $G_T$-torus. Since $x_1,\ldots,x_{12}\in
J_{_T}$ for all $T\in\mathcal{I},$ we get that $q_{48}(J_{_T})\neq0$
and so $J_{_T}$ is of Hermitian type. So $J$ is a direct union of
Jordan tori of Hermitian type.

We know that $J$ is special, so there is  an $*$-envelope  $\aa$ of
$J$ satisfying the following condition:

\parbox{5in}{\begin{center}   $J\sub H(\aa,*)$ and  if $I$ is an $*$-ideal of $A,$ then $I\cap J\ne \{0\}.$ \end{center}}

\noindent (see \cite{MZ}).  Also by Special Hermitian Structure
Theorem and Lemma \ref{q48}, the associative  subalgebra $P$ of $\aa$
generated by $J$ is $*$-prime and $J=H(P,*).$  Now if  for $T\in\i,$
$\pp_{_T}$ {is the associative subalgebra of $\aa$ generated by $J_{_T},$
we have}
$J_{_T}=H(\pp_{_T},*).$ We also have $\pp=\cup_{T\in \i}\pp_{_T}.$

We next note that for $T\in\i,$ $G_{_T}$ is a finitely generated
torsion free abelian group and so it is a free abelian group of
finite rank. Now we get the result using Proposition
\ref{classification}.
\end{proof}

\section{\bf {Jordan tori of Clifford type}}\label{clifford-type}\setcounter{equation}{0}
Let $R$ be a unital commutative associative ring and $V$ be an
$R$-module. Let $\fm:V\times V\longrightarrow R$ be a symmetric $R$-
bilinear form. Define a linear $R$-algebra structure on $J:=R 1\oplus
V$ by having $1$ as the identity element and requiring $v\cdot
w=(v,w)1$, for $v,w\in V$. Then $J$ is a Jordan algebra called the
{\it Jordan algebra of the bilinear form $\fm$} or a {\it Jordan spin
factor}. {We recall that a Jordan algebra is called of Clifford type if its central closure
is a Jordan algebra of a symmetric bilinear form.}

{The following example is a generalization of a Clifford torus
appeared in {\cite[Theorem, III.2.9]{AABGP}} as the coordinate algebra of an
extended affine Lie algebra of type $A_1$. The setting is based on
\cite[Example 5.2]{Yo1} and {\cite{Yo2}.}}

{\begin{DEF}\label{exanew}
{\em
Let $G$ be an abelian group, $S$ a pointed reflection subspace of $G$ and $\Gamma$ a subgroup of $G$ such that
\begin{equation}\label{ipm1} 2G\sub\Gamma\subsetneq S\sub G\andd
S+\Gamma=S. \end{equation}
Let $I$ be a set of coset
representatives for $\{\sg+\Gamma\mid \sg\in
S\}\setminus\{\Gamma\}$. Then for a collection $\{a_\ep\}_{\ep\in I}$, $a_\ep\in\bbbf^\times$, we call the triple $(S,\Gamma,\{a_\ep\})$ a {\it Clifford triple}.
}
\end{DEF}}

\begin{exa}\label{clifford2}
\em{Let $G$ be an abelian group, not necessarily torsion free, and let
$(S,\Gamma,\{a_\ep\})$ be a Clifford triple.
Let $Z$ be the group algebra of
$\Gamma$ over $\bbbf$, namely $Z:=\bbbf[\Gamma]=\bigoplus_{\gamma\in
\Gamma}\bbbf z^\gamma$ with $z^\gamma z^\tau=z^{\gamma+\tau}$, for
$\gamma,\tau\in\Gamma$.
 Let $V$ be a free $Z$-module
with basis $\{t_\ep\}_{\ep\in I}$. Define a $Z$-bilinear form
$f:V\times V\longrightarrow Z$ by $Z$-linear extension of
\begin{equation}\label{ipm5}
f(t_\ep,t_\eta)=\left\{\begin{array}{ll}
a_\ep z^{2\ep}&\hbox{if }\ep=\eta,\\
0&\hbox{otherwise,}\end{array}\right.
\end{equation}
for all $\ep,\eta\in I$ (here we
note that $2\ep\in\Gamma$ by (\ref{ipm1})). Let
$$J:=J(S,\Gamma,\{a_\ep\}_{\ep\in I})=Z\oplus V$$
be the Jordan algebra over $Z$ of $f$. We note that
$V=\bigoplus_{\ep\in I}Zt_\ep=\bigoplus_{\ep\in
I,\gamma\in\Gamma}\bbbf z^\gamma t_\ep$.
We also note that for $\sg\in S$, there exists a unique $\ep_\sg\in
I\cup\{0\}$ such that $\sg-\ep_\sg\in\Gamma$. Set $t_0:=1\in J$. Now
for $\sg\in G$, we set
$$J_\sg:=\left\{\begin{array}{ll}\bbbf z^{\sg-\ep_\sg}t_{\ep_\sg}&\hbox{if }\sg\in S,\\
0&\hbox{otherwise}.
\end{array}\right.
$$
Then $J=\bigoplus_{\sg\in G}J_\sg$ and $\supp(J)=S$.

We next show that $J$ is $G$-graded. Let
$\sg,\tau\in S$. If $\ep_\sg=\ep_\tau=0$, then
$$J_\sg J_\tau=\bbbf z^\sg z^\tau=\bbbf z^{\sg+\tau}=J_{\sg+\tau}.$$
If $\ep_\sg=0$ and $\ep_\tau\not=0$, then
$$J_\sg J_\tau=\bbbf z^\sg z^{\tau-\ep_\tau}t_{\ep_\tau}=\bbbf
z^{\sg+\tau-\ep_\tau}t_{\ep_\tau}=J_{\sg+\tau}.$$ Finally, suppose
$\ep_\sg\not=0$ and $\ep_\tau\not=0$. We note that if
$\ep_\sg=\ep_\tau$, then {$\sg+\tau\in\Gamma\sub S$} and $J_{\sg+\tau}=\bbbf
z^{\sg+\tau}$. Then
\begin{eqnarray*}
J_\sg J_\tau=\bbbf
z^{\sg-\ep_\sg}t_{\ep_\sg}z^{\tau-\ep_\tau}t_{\ep_\tau}&=&\bbbf
z^{\sg+\tau-\ep_\sg-\ep_\tau}f(t_{\ep_\sg},t_{\ep_\tau})z^{2\ep_\sg}\\
&=&\left\{\begin{array}{ll}
\bbbf z^{\sg+\tau}&\hbox{if }\ep_\sg=\ep_\tau\\
0&\hbox{otherwise}.\end{array}\right.
\end{eqnarray*}
So $J_\sg J_\tau=
J_{\sg+\tau}$ if $\ep_\sg=\ep_\tau$, and $J_\sg J_\tau=\{0\}$
otherwise. This completes the proof that $J$ is a $G$-graded Jordan
algebra {over $Z$}. Thus $J$ is a Jordan $G$-torus with $Z(J)=Z$.
If $G$ is torsion free, then we can consider the central closure $\bar{J}$ of $J$.
If $\bar{V}:=\bar{Z}\otimes_Z V$, then $\bar{J}$ can be identified with $\bar{Z}\oplus\bar{V}$. Extending $f$ to $\bar{f}:\bar{V}\times\bar{V}\longrightarrow\bar{Z}$
by {$\bar{f}(\a\otimes v,\b\otimes w):=\a\b f(v,w)$}, one can see that $\bar{J}$ is the Jordan
algebra of the extended bilinear form $\bar f$. {Hence $J$ is a of Clifford type, which we call it the {\it Clifford $G$-torus}
associated to the Clifford triple $(S,\Gamma,\{a_{\ep}\})$.}
}
\qedexa
\end{exa}

{
\begin{thm}\label{clif}
{Let $G$ be a torsion free abelian group and $J$ be a Jordan $G$-torus of Clifford type with
support $S$ and central grading group $\Gamma$.
Let $I$ be a set of coset
representatives for $\{\sg+\Gamma\mid \sg\in
S\}\setminus\{\Gamma\}$. Then for each $\ep\in I$, there exists $a_\ep\in\bbbf^\times$ such that
$(S,\Gamma,\{a_\ep\})$ is a Clifford triple and $J$ is graded isomorphic to the Clifford $G$-torus $J(S,\Gamma,\{a_\ep\}_{\ep\in I})$ associated
to the Clifford triple $(S,\Gamma,\{a_\ep\})$.}
\end{thm}
}
\proof
 By assumption, the central closure
$\bar J=\bar{Z}\otimes_Z J$ is a Jordan algebra over $\bar Z$ of a
symmetric bilinear form, where $\bar Z$ is the filed of fractions of
the center $Z=Z(J)$ of $J$.
Thus $\bar J$ has degree less
than or equal $2$ over $\bar Z$, that is there exists a $\bar
Z$-linear form $\hbox{tr}:\bar J\longrightarrow\bar Z$ and a $\bar
Z$-quadratic map $n:\bar J\longrightarrow \bar Z$ with $n(1)=1$ such
that for all $x\in\bar J$,
$$x^2-\hbox{tr}(x)x+n(x)1=0.$$
Let {$n:\bar J\times\bar J\longrightarrow \bar Z$} be the symmetric
$\bar Z$-bilinear form associated to the quadratic map $n$. Let
$W:=\{x\in\bar J\mid\hbox{tr}(x)=0\}$. Then $\bar J=\bar Z 1\oplus W$
is the Jordan algebra over $\bar Z$ of the bilinear form
$$h:=-\frac{1}{2}n\fm_{\mid_{W\times W}}.
$$
If $\dim_{\bar Z}\bar J=1$, then by Lemma \ref{ipm6}(iii),
$\supp(J)=\Gamma=G$ and so $J=Z$. Hence $J$ is a commutative
associative torus and so is $G$-graded isomorphic to the group
algebra of $G$ over $\bbbf$.

We assume from now on that $\dim_{\bar Z}\bar J\not=1$.
The same
argument as in \cite[Claim 1]{Yo1} shows that
\begin{equation}\label{ipm2}
\begin{array}{c}
\hbox{tr}(\bar{J}_{\bar \a})=\{0\}\quad\hbox{for}\quad\a\in G\setminus\Gamma,\\
2G\sub\Gamma\subsetneq\hbox{supp(J)}\andd\supp(J)+\Gamma=\supp(J).
\end{array}
\end{equation}
Moreover,
\begin{equation}\label{imp3}
G/\Gamma\hbox{ {cannot} be a nontrivial cyclic group}.
\end{equation}

Recall from Lemma \ref{ipm6}(ii) that $J=\bigoplus_{\bar\a\in
G/\Gamma}J_{\bar\a}$ is a $G/\Gamma$-graded algebra over $Z$. Then
$\hbox{tr}(J_{\bar\a})\sub\hbox{tr}(\bar{J}_{\bar\a})=\{0\}$ for
$\bar\a\not=\bar{0}$, by \eqref{ipm2}. So $V:=\bigoplus_{\bar\a\not=\bar
0}J_{\bar\a}=\bigoplus_{\a\in G\setminus\Gamma}ZJ_\a\sub W.$ Then
$$J=\bigoplus_{\bar\a\in
G/\Gamma}J_{\bar\a}=\bigoplus_{0\not=\bar\a\in G/\Gamma}+J_{\bar
0}=V\oplus Z,$$ as a direct sum of $Z$-modules. For $x,y\in V$,
$x\cdot y=h(x,y)\cdot 1\in J\cap\bar{Z}\cdot 1=J\cap Z(\bar{J})=Z$.
Therefore, $J=Z\oplus V$ is the Jordan algebra over $Z$ of
$f:=h_{\mid_{V\times V}}$. Let $S:=\supp(J)$. By Lemma \ref{b3},
$S$ is a pointed reflection space in $G$. By (\ref{ipm2}), $\Gamma$ is
a proper subset of $S$ and the pair (S,$\Gamma$) satisfies
(\ref{ipm1}). Next let $I$ be a set of coset representatives for $\{\sg+\Gamma\setminus \sg\in S\}\setminus\{\Gamma\}$, namely
$$S=\bigcup_{\ep\in I\cup\{0\}}(\ep+\Gamma).
$$
For $\ep\in I$, let $0\not=t_\ep\in J_\ep.$ Then using Lemma
\ref{ipm6}(ii),
$V=\bigoplus_{\bar\a\not=0}J_{\bar\a}=\bigoplus_{\ep\in I}Zt_\ep$, as
direct sum of $Z$-modules. Since
$Z=\bigoplus_{\gamma\in\Gamma}J_\gamma$ is a commutative associative
$\Gamma$-torus, $Z$ can be identified with the group algebra
$\bbbf[\Gamma]$ of $\Gamma$ over $\bbbf$. If $\ep\not=\ep'\in I$, we
have $e+\ep'\not\in\Gamma$ (since $\ep$ and $\ep'$ are distinct coset
representatives of $\Gamma$ in $S$). Therefore,
$$t_\ep t_{\ep'}=f(t_\ep, t_{\ep'})\in J_{\overline{\ep+\ep'}}\cap
J_{\bar 0}=\{0\}.$$ Also $0\not=t^2_{\ep}=f(t_\ep,t_\ep)\in
J_{2\ep}=\bbbf z^{2\ep}$, say $f(t_\ep,t_\ep)=a_\ep$ for some
$0\not=a_\ep\in \bbbf$. (We note that $2\ep\in 2G\sub\Gamma.$) Now
since $V=\bigoplus_{\ep\in I}Zt_\ep$, it is clear that the bilinear
form $f$ here coincides with the one given in Example \ref{clifford2}
(see (\ref{ipm5})). {Thus $J$ is graded isomorphic to the Clifford $G$-torus $J(S,\Gamma,\{a_\ep\}_{\ep\in I})$ of Example \ref{clifford2} associated to $(S,\Gamma,\{a_\ep\})$.} \qed

\section{\bf Jordan tori of Albert type}\label{albert-type}\setcounter{equation}{0}
Throughout this section, we assume that $G$ is a torsion free abelian
group. 
{We recall that an Albert algebra is by definition
a form of a 27-dimensional central simple exceptional Jordan algebra
of degree 3.}
{We also recall that a Jordan torus of
Albert type is by definition a Jordan torus whose central closure is
an Albert algebra.}

\begin{DEF}\label{yoshii-6.4}\cite[Defintition 6.4]{Yo1}
{\em A prime Jordan or associative algebra $P$ over $\bbbf$ is said
to have {\it central degree} $3$, if the central closure $\bar P=\bar
Z\otimes_ZP$ is finite dimensional and has (generic) degree $3$.}
\end{DEF}

The following is a modified version of \cite[Proposition
6.7]{Yo1}. The proof follows exactly by the same reasoning as in
the proof of \cite[Propostion 6.7]{Yo1}.

\begin{pro}\label{6.7-Yoshii}
Let $G$ be a torsion free abelian group and $T=\bigoplus_{\a\in
G}T_\a$ be a Jordan or an associative $G$-torus over
$\bbbf$ of central degree $3$. Let $\tr$ be the generic trace of the
central closure $\bar T$, and let $\Gamma$ be the central grading
group of $T$. Then $3G\sub\Gamma\subsetneq G$ and $\supp(T)=G$.
Moreover for any $\a\in G\setminus\Gamma$, we have $\tr(T_\a)=\{0\}$.
\end{pro}

{The following example gives a construction of an associative algebra
which will be crucial in the classification of Jordan tori of Albert type. In what follows, for $n\in\bbbz_{\geq 0}$, we let
$\ep(n)\in\{0,1,2\}$ be the congruent mod $3$ of $n$ and
$\eta(n):=n-\ep(n)$.}

\begin{exa}\label{albert-tori}{
{\em Consider the pair $(G,\Gamma)$ where $G$ is a torsion free abelian
group and $\Gamma$ is a subgroup of $G$ satisfying $3G\sub\Gamma$
and $|G/\Gamma|=9$. Let
$\mu:\Gamma\times\Gamma\longrightarrow\bbbf^\times$ be a symmetric
$2$-cocycle. Assume $\bbbf$ contains
a primitive $3$-th root of unity $q$. We fix $\sg_1$ and
$\sg_2$ in $G$ such that $\{\sg_1+G,\sg_2+G\}$ is a basis for
$G/\Gamma$ over the field of $3$ elements. Then $G=\bigcup_{0\leq
i,j\leq 2}(i\sg_1+j\sg_2+\Gamma)$. {Define} $\lam:G\times
G\longrightarrow \bbbf^\times$ by
\begin{equation}\label{ipm-day}
\begin{array}{l}
\lam(i\sg_1+j\sg_2+\gamma,i'\sg_1+j'\sg_2+\gamma')= \\
\qquad\qquad q^{ji'}
\mu\big(\eta(i+i')\sg_1,\eta(j+j')\sg_2\big)\mu
(\gamma,\gamma')\mu\big(\eta(i+i')\sg_1+\eta(j+j')\sg_2,\gamma+\gamma'\big),
\end{array}
\end{equation}
for $0\leq i,j,i',j'\leq 2$, $\gamma,\gamma'\in\Gamma$.
We claim that $\lam$ is a $2$-cocycle on $G$.
To see this, we must show that for any three fixed elements $\sg:=i\sg_1+j\sg_2+\gamma$,
$\tau:=i'\sg_1+j'\sg_2+\gamma'$ and
$\delta:=i^{''}\sg_1+j^{''}\sg_2+\gamma^{''}$ of the above form, the $2$-cocycle identity \eqref{4ipm} holds,
namely
$$
q^{\ep(j+j')i^{''}}q^{ji'}A=q^{j\ep(i'+i^{''})}q^{j'i^{''}}B
$$
where
$$\begin{array}{c}
A:=\mu(\eta(\ep(i+i')+i^{''})\sg_1,\eta(\ep(j+j')+j^{''})\sg_2)\\
\mu(\gamma+\gamma'+\eta(i+i')\sg_1+\eta(j+j')\sg_2,\gamma^{''})\\
\mu(\eta(\ep(i+i')+i^{''})\sg_1+\eta(\ep(j+j')+j^{''})\sg_2,\gamma+\gamma'+\eta(i+i')\sg_1+\eta(j+j')\sg_2+\gamma'')\\
\mu(\eta(i+i')\sg_1,\eta(j+j')\sg_2)\mu(\gamma,\gamma')\mu(\eta(i+i')\sg_1+\eta(j+j')\sg_2,\gamma+\gamma'),
\end{array}
$$
and
$$\begin{array}{c}
B:=\mu(\eta(i+\ep(i'+i^{''}))\sg_1,\eta(j+\ep(j'+j^{''}))\sg_2)\\
\mu(\gamma'+\gamma^{''}+\eta(i'+i^{''})\sg_1+\eta(j'+j^{''})\sg_2,\gamma)\\
\mu(\eta(i+\ep(i'+i^{''}))\sg_1+\eta(j+\ep(j+j^{''}))\sg_2,\gamma'+\gamma^{''}+\eta(i'+i^{''})\sg_1+\eta(j'+j^{''})\sg_2
+\gamma)\\
\mu(\eta(i'+i^{''})\sg_1,\eta(j'+j^{''})\sg_2)\mu(\gamma',\gamma^{''})\mu(\eta(i'+i^{''})\sg_1+
\eta(j'+j^{''})\sg_2,\gamma'+\gamma^{''}).
\end{array}
$$
Since $q^{\ep(j+j')i^{''}}q^{ji'}=q^{j\ep(i'+i^{''})}q^{j'i^{''}}$, $\lam$ is a $2$-cocycle if and only if
$A=B$. Let
$$a:=\eta(\ep(i+i')+i^{''})\sg_1+\eta(\ep(j+j')+j^{''})\sg_2+
\eta(i+i')\sg_1+\eta(j+j')\sg_2+
\gamma+\gamma'+\gamma^{''}
$$
and
$$b:=
\eta(i+\ep(i'+i^{''}))\sg_1+\eta(j+\ep(j'+j^{''})\sg_2+
\eta(i'+i^{''})\sg_1+\eta(j'+j^{''})\sg_2+
\gamma'+\gamma^{''}+\gamma.
$$
Then in the commutative
associative torus $(\bbbf^t[\Gamma]:=\bigoplus_{\gamma\in\Gamma}\bbbf x^{\gamma},\mu)$, we have
$$\big(x^{\eta(\ep(i+i')+i^{''})\sg_1}x^{\eta(\ep(j+j')+j^{''})\sg_2}\big)
\big(x^{\eta(i+i')\sg_1}x^{\eta(j+j')\sg_2}\big)
(x^{\gamma}x^{\gamma'})x^{\gamma^{''}}=Ax^a
$$
and
$$\big(x^{\eta(i+\ep(i'+i^{''}))\sg_1}x^{\eta(j+\ep(j'+j^{''})}\big)
\big(x^{\eta(i'+i^{''})\sg_1}x^{\eta(j'+j^{''})\sg_2}\big)
(x^{\gamma'}x^{\gamma^{''}})x^\gamma=Bx^b.
$$
Therefore, if we show that $a=b$, then we get $A=B$ if and only if
\begin{equation}\label{tt}
\begin{array}{c}
\big(x^{\eta(\ep(i+i')+i^{''})\sg_1}x^{\eta(\ep(j+j')+j^{''})\sg_2}\big)
\big(x^{\eta(i+i')\sg_1}x^{\eta(j+j')\sg_2}\big)\\
\qquad\qquad\qquad=
\big(x^{\eta(i+\ep(i'+i^{''}))\sg_1}x^{\eta(j+\ep(j'+j^{''})}\big)
\big(x^{\eta(i'+i^{''})\sg_1}x^{\eta(j'+j^{''})\sg_2}\big)
,
\end{array}
\end{equation}
for any $0\leq i,i',i^{''},j,j',j^{''}\leq 2$. Now $a=b$ if and only if
$$\begin{array}{c}
\eta(\ep(i+i')+i^{''})\sg_1+\eta(\ep(j+j')+j^{''})\sg_2+
\eta(i+i')\sg_1+\eta(j+j')\sg_2+
\gamma+\gamma'+\gamma^{''}\\
=
\eta(i+\ep(i'+i^{''}))\sg_1+\eta(j+\ep(j'+j^{''})\sg_2+
\eta(i'+i^{''})\sg_1+\eta(j'+j^{''})\sg_2+
\gamma'+\gamma^{''}+\gamma,
\end{array}
$$
which in turn holds if and only if for any $i,i',i^{''}$,
\begin{equation}\label{rr}
\eta(\ep(i+i')+i^{''})+\eta(i+i')=\eta(i+\ep(i'+i^{''}))+\eta(i'+i^{''}).
\end{equation}
To see that this last equality holds, we note that
$$\ep(i+i')+i^{''}+\eta(i+i')=(i+i')+i^{''}=i+(i'+i^{''})=i+\ep(i'+i^{''})+\eta(i'+i^{''})$$ and so
\begin{eqnarray*}
\eta(\ep(i+i')+i^{''})+\eta(i+i')&=&
\eta\big(\ep(i+i')+i^{''}+\eta(i+i')\big)\\&=&
\eta\big(i+\ep(i'+i^{''})+\eta(i'+i^{''})\big)\\&=&
\eta(i+\ep(i'+i^{''}))+\eta(i'+i^{''}).
\end{eqnarray*}
Then \eqref{rr} holds and $a=b$. Thus $A=B$ if and only if \eqref{tt} holds. Now the
left hand side in \eqref{tt} is equal to
$$\mu(\eta(\ep(i+i')+i^{''})\sg_1,\eta(i+i')\sg_1)x^{c_1\sg_1}\mu(\eta(\ep(j+j')+j^{''})\sg_2,\eta(j+j')\sg_2)x^{c_2\sg_2}
$$where $c_1:=\eta(\ep(i+i')+i^{''})+\eta(i+i')$ and
$c_2:=\eta(\ep(j+j')+j^{''})+\eta(j+j').$ Also the right hand side in \eqref{tt} is equal to
$$\mu(\eta(i+\ep(i'+i^{''}))\sg_1,\eta(i'+i^{''})\sg_2)x^{c'_1\sg_1}\mu(\eta(j+\ep(j'+j^{''}))\sg_2,\eta(j'+j^{''})\sg_2)x^{c'_2\sg_2}
$$where $c'_1:=\eta(i+\ep(i'+i^{''})+\eta(i'+i^{''})$ and
$c'_2:=\eta(j+\ep(j'+j^{''}))+\eta(j'+j^{''}).$ By \eqref{rr}, $c_1=c_2$ and $c'_1=c'_2$.
So $A=B$ if and only if
$$\begin{array}{c}
\mu(\eta(\ep(i+i')+i^{''})\sg_1,\eta(i+i')\sg_1)\mu(\eta(\ep(j+j')+j^{''})\sg_2,\eta(j+j')\sg_2)
\\
\qquad\qquad\qquad=\mu(\eta(i+\ep(i'+i^{''})\sg_1),\eta(i'+i^{''})\sg_1)\mu(\eta(j+\ep(j'+j^{''})\sg_2),\eta(j'+j^{''})\sg_2)
\end{array}
$$
for all $0\leq i,i',i^{''},j,j',j^{''}\leq 2$. But clearly the latter holds if and only if,
\begin{equation}\label{uu}
\begin{array}{c}
\mu(\eta(\ep(i+i')+i^{''})\sg,\eta(i+i')\sg)=\mu(\eta(i+\ep(i'+i^{''})\sg,\eta(i'+i^{''})\sg),
\end{array}
\end{equation}
for all $\sg\in G$ and $0\leq i,i',i^{''}\leq 2$. Let $\a:=\ep(i+i')+i^{''}$, $\b:=i+i'$, $\a':=i+\ep(i'+i^{''})$ and $\b':=i'+i^{''}$.
Then $\eta(\a)+\eta(\b)=\eta(\a+\eta(\b))=\eta(i+i'+i^{''})$. Similarly, $\eta(\a')+\eta(\b')=\eta(i+i'+i^{''})$.
So $\eta(\a)+\eta(\b)=\eta(\a')+\eta(\b')$. From this and the fact that
$\eta(\a),\eta(\b),\eta(\a'),\eta(\b')\in\{0,3\}$, we get $\eta(\a)\eta(\b)=\eta(\a')\eta(\b')\in\{0,9\}$.
Therefore, \eqref{uu} holds. Hence $A=B$ and $\lam$ is a $2$-cocycle.

 We denote the {$2$-cocycle} $\lam$ by
$\lam:=\lam(q,\mu)$ and the corresponding associative
$G$-torus by $(\bbbf^t[G],\lam(q,\mu))_\Gamma$, and call it
the associative $G$-torus associated to the pair $(G,\Gamma)$. Let
$T=(\bbbf^t[G],\lam(q,\mu))_\Gamma$. Note that if we fix
$x_i:=x^{\sg_i}\in T_{\sg_i}$, $i=1,2$, and $x^\gamma\in
T^\gamma$, for each $\gamma\in\Gamma$, then the elements
$x_1^{i_1}x_2^{i_2}x^\gamma$, $0\leq i_1,i_2\leq 2$,
$\gamma\in\Gamma$ form a basis of $T$ over $\bbbf$. Moreover, we
have
$$(x^i_1x_2^jx^\gamma)(x_1^{i'}x_2^{j'}x^{\gamma'})=\lam(i\sg_1+j\sg_2+\gamma,i'\sg_1+j'\sg_2+\gamma')x_1^{\ep(i+i')}x_2^{\ep(j+j')}x^{\gamma''},$$
where $0\leq i,j,i',j'\leq 2$, $\gamma,\gamma'\in\Gamma$ and
$\gamma''=\gamma+\gamma'+\eta(i+i')+\eta(j+j')$. Using this, it is
easy to see that
\begin{equation}\label{central-ipm}
\begin{array}{c}
Z(T)=\bigoplus_{\gamma\in\Gamma}\bbbf x^\gamma,\quad x_2x_1=q x_1x_2,\quad\hbox{and}\vspace{2mm}\\
x_1^{i}x_2^{j}\in Z(T)\quad\hbox{for all $i,j\in\bbbz$ with }\quad
i\equiv j\equiv 0\hbox{ mod }3.
\end{array}
\end{equation}
It follows that the central closure of $T$ is  $9$-dimensional,
namely $\bar{J}:=\bar Z\otimes_Z T\cong \bigoplus_{0\leq i,j\leq
2}\bbbf x_1^ix_2^j$. Since $\bar T$ is domain, it is a division
algebra and so is an associative algebra of central degree $3$
(see Definition \ref{yoshii-6.4}). Note that $Z(T^+)=Z(T)$, so
$$\overline{T^+}=T^+\otimes_{Z(T^+)}\overline{Z(T^+)}={T^+}\otimes_{Z(T)}\overline{Z(T)}=\overline{T}^{^+}.
$$
So $\overline{T^+}$ is a nine dimensional central special Jordan
division algebra over $\overline{Z(T)}$. Hence by Lemma
\cite[Lemma 2.11]{Yo1}, it has degree $3$.}}
\qedexa
\end{exa}

{The following proposition gives a characterization of associative $G$-tori of central degree $3$.}

\begin{pro}\label{yoshii-6.12} Let $G$ be a torsion free abelian group and $T$ be an associative $G$-torus over
$\bbbf$ with central grading group $\Gamma$. Then $T$ has central
degree $3$ if and only if $3G\sub \Gamma$, $\supp(T)=G$ and $G/\Gamma$ is a vector
space of dimension $2$ over the filed of $3$ elements. If $T$ has
central degree $3$, then $\bbbf$ contains a primitive third root of
unity, say $\omega$, and $T\cong (\bbbf^t[G],\lam(\omega,\mu))_\Gamma$ where $\mu$
is a symmetric $2$-cocycle on $\Gamma$.
Moreover, if $\Gamma$ is free abelian or $\bbbf$ is algebraically
closed then $T\cong (\bbbf^t[G],\lam(\omega,1))$.

Conversely, suppose $\bbbf$ contains a primitive third root of
unity $\omega$. Also suppose $G$ is a torsion free abelian group
and $\Gamma$ is a subgroup satisfying $3G\sub\Gamma$ and
$|G/\Gamma|=9$. Let $\mu$ be a symmetric $2$-cocycle on $\Gamma$. Then
$(\bbbf^t[G],\lam(\omega,\mu))_\Gamma$ is an associative $G$-torus
of central degree $3$ with central grading group $\Gamma$.
\end{pro}

\proof Let $T=\bigoplus_{\a\in G}T^\a$ be an associative torus
over $\bbbf$ of central degree $3$ and let $\bar T$ be its central
closure over $\bar Z$. By Proposition \ref{6.7-Yoshii}, $\supp(T)=G$ and $G/\Gamma$
is a nontrivial vector space over the filed of $3$ elements. By
Lemma \ref{ipm6}(iii), we have $\dim_{\bar Z}{\bar T}=|G/\Gamma|$.
Since by definition $\bar T$ is finite dimensional over $\bar Z$,
we have $\dim_{\bar Z}\bar T=3^m$ for some positive integer $m$.
Now $\bar T$ as a finite dimensional associative domain is a
division algebra, by Wedderburn's structure theorem. So as  $\bar T_{\bar Z}$ is a
central simple associative algebra with $\dim\bar T_{\bar Z}=3^m,$ we have $m=2$. It is also
clear that an associative torus whose central grading group
$\Gamma$ satisfies $|G/\Gamma|=9$ has central degree $3$. In fact
$\bar T$ has dimension $9$ over $\bar Z$ and is a division
associative algebra so by Lemma \cite[Lemma 2.11]{Yo1} it has degree
$3$.


Next, we assume that $T=\bigoplus_{\a\in G}T^\a$ is an
associative torus whose central grading group satisfies $3G\sub
\Gamma\subsetneq G$, $|G/\Gamma|=9$ and $\supp(T)=G$. We fix
$\sg_1$, $\sg_2$ in $G$ such that $\{\sg_i+\Gamma\mid i=1,2\}$ is a basis
for the vector space $G/\Gamma$. Then $G=\bigcup_{0\leq i,j\leq
2}(i\sg_1+j\sg_2+\Gamma)$. We fix $x_i:=x^{\sg_i}\in T^{\sg_i}$,
$i=1,2$ and $x^\gamma\in
T^\gamma$, for each $\gamma\in\Gamma$. Then $x_1x_2\not=x_2x_1$ and the elements $x_1^{i_1}x_2^{i_2}x^\gamma$, $0\leq
i_1,i_2\leq 2$, $\gamma\in\Gamma$ form a basis for $T$ over
$\bbbf$. Moreover, as $3G\sub\Gamma$,
\begin{equation}\label{central-ipm}
(x_1^{i_1}x_2^{i_2})^3x^\gamma\in Z(T)\quad\hbox{for all}\quad
0\leq i_1,i_2\leq 2\hbox{ and }\gamma\in\Gamma.
\end{equation}
Since $x_1x_2,x_2x_1\in T^{\sg_1+\sg_2}$, there exists $q\in\bbbf^\times$ such that $x_2x_1=qx_1x_2$.
Then as
$x_1^3$ is central, we get $q^3=1$.
Thus $\bbbf$ must contain a primitive third
root of unity, say $\omega$. Then $q=\omega$ or $\omega^2$.  Let $\lam:G\times
G\longrightarrow\bbbf^\times$ be the corresponding $2$-cocycle for
$T$ with respect to the basis mentioned above.
Then we have,
$$(x^i_1x_2^jx^\gamma)(x_1^{i'}x_2^{j'}x^{\gamma'})=\lam(i\sg_1+j\sg_2+\gamma,i'\sg_1+j'\sg_2+\gamma')x_1^{\ep(i+i')}x_2^{\ep(j+j')}x^{\gamma''},$$
where $0\leq i,j,i',j'\leq 2$, $\gamma,\gamma'\in\Gamma$,
$\gamma''=\gamma+\gamma'+\eta(i+i')+\eta(j+j')$, and $\ep$ and
$\eta$ are defined as in Example \ref{albert-tori}. Denote by
$\mu:\Gamma\times\Gamma\longrightarrow\bbbf^\times$ the symmetric
$2$-cocycle obtained from $\lam$ by restriction to $\Gamma$.  Then
using (\ref{central-ipm}) and the facts that $x_2x_1=q
x_1x_2$ and $\eta(n)G\sub 3G\sub\Gamma$ for all $n\in\bbbz$, we
see that
\begin{equation}\label{ipm-day}
\begin{array}{l}
\lam(i\sg_1+j\sg_2+\gamma,i'\sg_1+j'\sg_2+\gamma')= \\
\qquad\qquad q^{ji'}
\mu\big(\eta(i+i')\sg_1,\eta(j+j')\sg_2\big)\mu
(\gamma,\gamma')\mu\big(\eta(i+i')\sg_1+\eta(j+j')\sg_2,\gamma+\gamma'\big),
\end{array}
\end{equation}
for $0\leq i,j,i',j'\leq 2$, $\gamma,\gamma'\in\Gamma$.
 Then, in the notation of Example \ref{albert-tori}, we have
$T=(\bbbf^t[G],\lam(q,\mu))_\Gamma$. But
one can see that the corresponding associative tori  $(\bbbf^t[G],\lam(\omega,\mu))_\Gamma$ and $(\bbbf^t[G],\lam(\omega^2,\mu))_\Gamma$ are
isomorphic, under the isomorphism induced by
$x_1^{i_1}x_2^{i_2}x^\gamma\mapsto x_2^{i_1}x_1^{i_2}x^\gamma$. So
we may assume that $q=\omega$, namely $T=(\bbbf^t[G],\lam(\omega,\mu))_\Gamma$.
We remind from \cite[Lemma 1.1]{OP} that if $\Gamma$ is free
abelian or $\bbbf$ is algebraically closed, then any commutative
twisted group algebra on $\Gamma$ is isomorphic to the commutative
untwisted group algebra. Thus if $\Gamma$ is free abelian or
$\bbbf$ is algebraically closed then $\mu$ can be taken to be $1$.
The converse part follows from Example \ref{albert-tori}
\qed

\begin{rem}\label{rem-ipm-90}
In the notation of Proposition \ref{yoshii-6.12}, let $G$ be free abelian of rank $\geq 2$ with a basis indexed in a
set, say $J$. Assume, $1,2\in J$. By Proposition \ref{yoshii-6.12},
$T\cong(\bbbf^t[G],\lam(\omega,1))$. However, by Example
\ref{gday-1}, we may assume $\lam(\omega,1)=\bq_\omega$, where
$\bq_\omega=(q_{ij})_{i,j\in J}$ is the quantum matrix satisfying
\begin{equation}\label{ipm-91}
q_{ij}=\left\{\begin{array}{ll} \omega&\hbox{if }i=1,j=2,\\
\omega^{-1}&\hbox{if }i=2,j=1,\\
1&\hbox{otherwise}.\end{array}\right.
\end{equation}
\end{rem}

{Using our earlier results and a modified reasoning of \cite[Proposition 6.13]{Yo1}, we get the following.
To be precise, we provide details of the proof.}

\begin{pro}\label{yoshii-6.13}
Let $\omega$ be a third root of unity. Let $J$ be a special Jordan {$G$-torus} over $\bbbf$ of central degree
$3$ with central grading group $\Gamma$. Then $3G\sub\Gamma\subsetneq G$ and $|G/\Gamma|=9$. Also,
$$J\cong_G\left\{\begin{array}{ll}
(\bbbf^t[G],\lam(\omega,\mu))_\Gamma^+&\hbox{if }\omega\in\bbbf\\
H((\bbbe^t[G],\lam(\omega,\mu))_\Gamma,\sg)&\hbox{if }\omega\not\in\bbbf,\end{array}
\right.
$$
 where $\mu$ is a $2$-cocycle on
$\Gamma$, $\bbbe=\bbbf(\omega)=\bbbf(\sqrt{-3})$ and
$\sg$ is the unique non-trivial Galois automorphism of $E$. \end{pro}

\proof Since $J$ is special, it is either a Hermitian torus
or a Clifford torus. We have already seen that if $J$
is Clifford torus then $\deg(\bar J)\leq 2$ (see \S
\ref{clifford-type}).  So $J$ can only be a Hermitian torus. By
Proposition \ref{6.7-Yoshii}, $\supp(J)=G$. Therefore by Theorem
\ref{92-1}, we have one of the {following} three possibilities:

- $J\cong H((\bbbf^t[G],\lam),\theta_q)$, $\lam$ a $2$-cocycle and $q$ a quadratic map,

- $J\cong (\bbbf^t[G],\lam)^+$, $\lam$ a $2$-cocycle,

- {$J\cong H((\bbbe^t[G],\lam),\theta)$, $\bbbe$ a quadratic
field extension of $\bbbf,$ $\lam$ a $2$-cocycle and $\theta$ an
involution, as defined in Lemma \ref{4.9 yoshii}}.

We begin by showing that the first possibility, considering it as an
identification, {cannot} happen. Consider the center $Z$ of
$J=H((\bbbf^t[G],\lam),\theta_q)$. By Proposition \ref{6.7-Yoshii},
$3G\sub\Gamma\subsetneq G=\supp(J)$. But as $q$ is a quadratic map,
we have $(x^\sg)^2$ is central for any $\sg\in G$ implying that
$2G\sub\Gamma$. Now $2G\cup 3G\sub \Gamma$ implies $\Gamma=G$ which
is absurd.

We now consider the second and the third possibilities. By definition,
$\bar J$ is a finite dimensional central special Jordan division
algebra over $\bar Z$ of degree $3$. By \cite[2.11]{Yo1} and Proposition \ref{ipm6}(iii), we have
$\dim_{\bar Z}\bar J=9$ and $G/\Gamma$ is a $2$ dimensional vectors
space over the field of $3$ elements.

If $J\cong_G (\bbbf^t[G],\lam)^+$, $\lam$ a $2$-cocycle, then taking
this as an identification, we get
$Z(J)=Z((\bbbf^t[G],\lam)^+)=Z((\bbbf^t[G],\lam))$ and so $\Gamma$ is
the central grading group of $(\bbbf^t[G],\lam)$. Then by Proposition
\ref{yoshii-6.12}, $\bbbf$ contains a primitive third root of unity
$\omega$ and $(\bbbf^t[G],\lam)\cong
(\bbbf^t[G],\lam(\omega,\mu))_\Gamma$, where $\mu$ is a symmetric
$2$-cocycle on $\Gamma$. Thus $J\cong (\bbbf^t[G],\lam(\omega,\mu))_\Gamma^+$

Finally, we suppose that the third possibility holds and we take it
as an identification. Then
$$Z\big((\bbbe^t[G],\lam)\big)=Z\big((\bbbe^t[G],\lam)^+\big)=Z(J\otimes_\bbbf
\bbbe)\cong Z(J)\otimes_\bbbf\bbbe.$$ So
$(\bbbe^t[G],\lam)$ is an associative $G$-torus with central
grading group $\Gamma$ such that $3G\sub\Gamma\subsetneq G$ and $G/\Gamma$ is a $2$ dimensional
vector space over $\bbbz_3$. Then by Proposition
\ref{yoshii-6.12}, $(\bbbe^t[G],\lam)$ has central degree $3$ and $\bbbe$ contains a
primitive third root of unity $\omega$ such that $(\bbbe^t[G],\lam)\cong (\bbbe^t[G],\lam(\omega,\mu))_\Gamma$, where $\mu$
is a $2$-cocycle on $\Gamma$. {It follows from Lemma \ref{4.9 yoshii} that
$\theta(x_i)=x_i$ for $i=1,2$ and that $\theta$ acts as an anti-automorphism on $(\bbbe^t[G],\lam(\omega,\mu))_\Gamma$. Therefore
$x_1x_2=\theta(x_2x_1)=\theta(\omega x_1x_2)=\theta(\omega)\omega x_1x_2$.} Thus $\theta(\omega)=\omega^{-1}\not=\omega$
and so $\omega\not\in\bbbf$.
Finally, as $[\bbbe:\bbbf]=2$, we have
$\bbbe=\bbbf(\omega)=\bbbf(\sqrt{-3}).$\qed

\begin{DEF}\label{albert}
{\em Let $G$ be a torsion free abelian group and $\Delta$, $\Gamma$ be two subgroups of $G$ satisfying
$$3G\subsetneq \Gamma\sub\Delta\sub G,\quad \dim_{\bbbz_3}(G/\Gamma)=3,\andd\dim(\Delta/\Gamma)=2.$$
Then we call the triple $(G,\Delta,\Gamma)$ an {\it Albert triple}.}
\end{DEF}

\begin{exa}\label{yoshii-6.8-2}
{\em Let $(G,\Delta,\Gamma)$ be an Albert triple.
We take $\sg_1,\sg_2,\sg_3\in G$ such that $\{\sg_i+\Gamma\mid 1\leq i\leq 3\}$ is a basis for $G/\Gamma$ and
$\{\sg_i+\Gamma\mid1\leq i\leq 2\}$ is a basis for $\Delta/\Gamma$. Then
$$G=\bigcup_{0\leq i,j,k\leq 2}(i\sg_1+j\sg_2+k\sg_3+\Gamma)\andd
\Delta=\bigcup_{0\leq i,j\leq 2}(i\sg_1+j\sg_2+\Gamma).$$ Let
$\aa:=(\bbbf^t[\Delta],\lam(\omega,\mu))_\Gamma=\bigoplus_{\sg\in\Delta}\aa^\sg$ be the
$\Delta$-tori associated to the pair
$(\Delta,\Gamma)$ (see Example \ref{albert-tori}), where $\mu$ is a $2$-cocycle on $\Gamma$
and $\omega$ is a third root of unity. Let $Z=Z(\aa)$,
and $\tr$ be
the generic trace of the central closure $\bar{\aa}$. We fix
nonzero elements $u_1\in \aa^{\sg_1}$, $u_2\in \aa^{\sg_2}$ and
$u_3\in \aa^{3\sg_3}$. We note that
$\aa$ is a free
 $Z$-module with free basis $\{u_1^{i}u_2^{j}\mid 0\leq i,j\leq 2\}$. Since $\tr$ is $Z$-linear, for any $z\in Z$ and a basis element $u_1^iu_2^j$,
 we have $\tr(u_1^iu_2^jz)=z\tr(u_1^iu_2^j)=0$ if $(i,j)\not=(0,0)$, by Proposition \ref{6.7-Yoshii}, and so
 $\tr(\aa)\sub Z$. Since $u_3$ is an invertible element of $Z$, we consider the first Tits construction
 $\bbba_t=(\aa,u_3)$, (see \cite[6.5]{Yo1}). We call $\bbba_t$ the {\it Jordan algebra associated to the Albert triple
 $(G,\Delta,\Gamma)$.}

{\bf Claim.} $\bbba_t$ is a Jordan $G$-torus of strong type.

To see this, we first give a $G$-grading to $\bbba_t$ as follows.
Recall that  $u_i\in \aa^{\sg_i}$ for $i=1,2$ and $u_3\in
\aa^{3\sg_3}$. We now fix $u_0=1\in\bbbf= \aa^0$ and nonzero elements
$u_\gamma\in \aa^\gamma$ for $\gamma\in\Gamma\setminus\{0,3\sg_3\}$. For
$\a=i\sg_1+j\sg_2+\gamma\in \Delta$, $0\leq i,j\leq 2$,
$\gamma\in\Gamma$, we set $u_\a:=u_1^iu_2^ju_\gamma$. Then
$\aa=\bigoplus_{\a\in\Delta}\bbbf u_\a$. Next for
$\a=i\sg_1+j\sg_2+k\sg_3+\gamma\in G$, $0\leq i,j,k\leq 2$,
$\gamma\in\Gamma$, we set
$$t_\a:=\left\{\begin{array}{ll}
(u_\a,0,0)&\hbox{if }k= 0\\
(0,u_{\a-\sg_3},0)&\hbox{if }k=1\\
(0,0,u_{\a+\sg_3})&\hbox{if }k=2.
\end{array}\right.
$$
We have $t_{\sg_3}=(0,1,0)$, $t_{2\sg_3}=(0,0,u_3)$ and
$t_{-\sg_3}=t^{-1}_{\sg_3}=(0,0,1)$. One easily checks that, as a
vector space we have $\bbba_t=\oplus_{\a\in G}\bbbf t_\a$. Moreover, considering the multiplication rule in $\bbba_t$, it is not hard,
even though tedious,
to see that
$\bbba_t$ is strongly $G$-graded as a Jordan algebra and so $\bbba_t$ is a $G$-torus of strong type. To be more precise on this, we give a rough argument
as follows. Let us
recall that as a vector space we have $\bbba_t=\aa\oplus\aa\oplus \aa$. Now for $a\in\aa$, we set
 $a^{(0)}:=(a,0,0)$, $a^{(1)}:=(0,a,0)$ and $a^{(2)}:=(0,0,a)$. Also for $\a=i\sg_1+j\sg_2+k\sg_3+\gamma\in G$ of the above form, we set
$(\a):=0$ if $k=0$, $(\a):=-1$ if $k=1$ and $(\a):=1$ if $k=2$. Then we have $t_\a=u_{\a+(\a)\sg_3}^{(k)}$. Now if $\a'=i'\sg_1+\j'\sg_2+k'\sg_3+\gamma'$ is another
element in $G$ of the above form, then it is easy to see that
$u^{(k)}_{\a+(\a)\sg_3}\times u^{(k')}_{\a'+(\a')\sg_3}=ru^{(k)}_{\a+(\a)\sg_3}\cdot u^{(k')}_{\a'+(\a')\sg_3}$ for some $s\in\bbbz/2$.
Therefore,
\begin{eqnarray*}
t_\a t_{\a'}=r\big(u_3^{(\a)(\a')(\a+\a')}u^{(k)}_{\a+(\a)\sg_3}\cdot u^{(k')}_{\a'+(\a')\sg_3}\big)^{(\ep(k+k'))}.
\end{eqnarray*}
But $u_3^{(\a)(\a')(\a+\a')}u^{(k)}_{\a+(\a)\sg_3}\cdot u^{(k')}_{\a'+(\a')\sg_3}$ is a homogenous element of degree
$$3(\a)(\a')(\a+\a')\sg_3+\a+\a'+(\a)\sg_3+(\a')\sg_3=\a+\a'+(\a+\a')\sg_3.$$
It follows that
$$t_\a t_{\a'}=ru^{\ep(k+k')}_{\a+\a'+(\a+\a')}=rt_{\a+\a'},$$
for some scalar $r$. This shows that $\bbba_t$ is $G$-graded. To see that it is of strong type, we need to show that
$r$ is nonzero or equivalently $a\times b\not=0$ if $a:=u^{(k)}_{\a+(\a)\sg_3}$ and $b:=u^{(k')}_{\a'+(\a')\sg_3}.$
Suppose to the contrary that $a\times b=0$. Then we must have $\tr(a\cdot b)=\tr(a)\tr(b)$. Now if both $a$ and $b$ are central
this gives $ab=3ab$, as $\tr(1)=3$, which is absurd. If $a$ is central but $b$ not, then we get $\tr(a\cdot b)=0$, which in turn implies
$ab=0$ which is again absurd. Finally, if both $a$ and $b$ are non-central, then again we get $\tr(a\cdot b)=0$ which together with $a\times b=0 $ implies $a\cdot b=0$ or
equivalently $ab=-ba$.
Then $ab=-ba=\omega^t ba$ for some integer $t$ which is absurd as $\omega$ is a third root of unity.

By \cite[Lemma 6.5]{Yo1}, the central closure $\bar{\bbba}_t$ of $\bbba_t$ is an Albert algebra over $\bar{Z}$,
and so $\bbba_t$ is a {Jordan $G$-torus of  Albert type. We refer to $\bbba_t$ as an {\it Albert $G$-torus} constructed from an Albert triple
$(G,\Delta,\Gamma)$}.\qedexa
}
\end{exa}

\begin{thm}\label{yoshii-thm-6.16}
{Let $J$ be a Jordan $G$-torus of Albert type over $\bbbf$ with central grading group $\Gamma$.
Then $G$ contains a subgroup $\Delta$ such that $(G,\Delta,\Gamma)$ is an Albert triple and
$J$ is graded isomorphic to the Albert $G$-torus $\bbba_t$, constructed from the Albert triple $(G,\Delta,\Gamma)$
(see Example \ref{yoshii-6.8-2}).
Conversely, given an Albert triple $(G,\Delta,\Gamma)$, the associated Jordan algebra $\bbba_t$ is an Albert $G$-torus.}
\end{thm}

\proof  Let $J=\bigoplus_{\sg\in
G}J^\sg$ be a Jordan $G$-torus as in the statement. Then the central
closure $\bar J$ is an Albert algebra over $\bar Z$, $Z:=Z(J)$. We recall that
an Albert algebra is a $27$-dimensional central simple exceptional
Jordan algebra of degree $3$. By Proposition \ref{6.7-Yoshii}, $3G\subsetneq
\Gamma\sub G$ and $\supp(J)=G$. Moreover, by Lemma \ref{ipm6},
$27=\dim_{\bar Z}{\bar J}=|G/\Gamma|$. Since $G/\Gamma$ is a vector
space over the field of $3$ elements, we have
$\dim_{\bbbz_3}G/\Gamma=3$. We fix $\sg_1,\sg_2,\sg_3\in G$ such that
$\{\sg_i+\Gamma\mid i=1,2,3\}$ is a basis for $G/\Gamma$. Then
$G=\bigcup_{0\leq i,j,k\leq 2}(i\sg_1+j\sg_2+k\sg_g+\Gamma)$. Set
$$\Delta:=\bigcup_{1\leq i,j\leq 2}(i\sg_1+j\sg_2+\Gamma)\andd
U:=\bigoplus_{\sg\in\Delta}J_\sg.$$ Since $3G\sub\Gamma$, $\Delta$ is
a subgroup of $G$ and so $U$ is a subalgebra of $J$. We now
show that $Z(U)=Z(J)$. Since $\Gamma\sub\Delta$, we have $Z(J)\sub Z(U)$.
Thus we must show $Z(U)\sub Z(J)$. Let $\Delta_1$ be the central
grading group of $U$. Then
$$3G\sub\Gamma\sub\Delta_1\sub\Delta\sub G.$$
Now using the same argument as in  \cite[Last
paragraph of page 163]{Yo1}, we see that $\Delta_1\subsetneq\Delta$.
By Lemma \ref{ipm6}, the central closure $\bar
U:=\overline{Z(U)}\otimes_{Z(U)}U$ is $(\Delta/\Delta_1)$-graded and
$\Delta/\Delta_1$ {cannot} be a non-trivial cyclic group. Thus
$2\leq\dim\Delta/\Delta_1\leq \dim\Delta/\Gamma=2$. This gives
$\dim\Delta/\Delta_1=2$ and $\Delta_1=\Gamma$. That is $Z(U)=Z=Z(J)$.

Since $Z(U)=Z(J)$, we have $\bar U=\bar Z\otimes_Z
U\hookrightarrow\bar J$. By \cite[2.6(ii)]{Yo1}, $\bar U$ is
central. Thus $\bar U$ is a central subalgebra of the division
algebra $\bar J$ and is $9$-dimensional as $|\Delta/\Gamma|=9$. So by
the classification of finite dimensional central simple Jordan
algebras $\bar U$ is special (see \cite[Corollary 2, pages 204 and
207]{J2}). Then by \cite[2.11]{Yo1}, $\bar U$
has degree $3$. Thus $U$ is a special Jordan $G$-torus of central
degree $3$. So we may use the characterization given in Proposition \ref{yoshii-6.13} for $U$, in terms
of a primitive third root of unity $\omega$ and a $2$-cocycle $\mu$ on $\Gamma$, namely, $U\cong_G (\bbbf^t[\Delta],\lam(\omega,\mu))_\Gamma^+$ if $\omega\in\bbbf$ and
$U\cong_G H((\bbbe^t[\Delta],\lam(\omega,\mu))_\Gamma,\sg)$ if $\omega\not\in\bbbf$, where $\bbbe=\bbbf(\omega)$ and
$\sg$ is the non-trivial Galois automorphism of $E$.

We assume first that $\omega\in\bbbf$. Then
$U=(\bbbf^t[\Delta],\lam(\omega,\mu))_\Gamma^+$. We fix nonzero elements
$u_1:=u_{\sg_1}\in J^{\sg_1}$, $u_2:=u_{\sg_2}\in J^{\sg_2}$ and
$x\in J^{\sg_3}$. Set $u_3:=u_{3\sg_3}:=x^3\in J^{3\sg_3}$. Let
$\hbox{tr}$ be the generic trace of $\bar{J}$. We have
{$$U=\bigoplus_{\sg\in\Delta}J^\sg=\bigoplus_{0\leq i,j\leq
2,\;\gamma\in\Gamma}J^{i\sg_1+j\sg_2+\gamma}
=\bigoplus_{0\leq i,j\leq 2,\;\gamma\in\Gamma}J^\gamma
J^{i\sg_1+j\sg_2}=\bigoplus_{0\leq i,j\leq 2}ZJ^{i\sg_1+j\sg_2}.$$}
Thus $U$ is a free $Z$-module with basis $\{u_1^iu_2^j\mid 0\leq
i,j\leq 2\}$. Now for $z\in Z$ and $0\leq i,j\leq 2$,
$\hbox{tr}(zu_{1}^{i}u_2^j)=z\hbox{tr}(u_1^iu_2^j)=0$ if $(i,j)\not=(0,0)$, by Proposition \ref{6.7-Yoshii}, and is
equal to $z\hbox{tr}(1)$ if $i=i=0$. Thus $\hbox{Tr}(\bbbf^t[\Delta],\lam(\omega,\mu))_\Gamma\sub Z$.

Since $x^3=u_3$ is an invertible element of $Z$, we may consider the
first Tits construction $\bbba_t:=(\aa,u_3)$ over $Z$, where
$\aa:=(\bbbf^t[\Delta],\lam(\omega,\mu))_\Gamma$ (see \cite[6.5]{Yo1}).
As we have seen in Example \ref{yoshii-6.8-2}, $\bbba_t$ is a Jordan $G$-torus of strong type.

Next, let
$$U^{\perp}:=\{y\in J\mid \hbox{Tr}(Uy)=0\}.$$
We show that $J^{\sg_3}$,$J^{2\sg_3}\sub U^\perp$.
Now for $0\leq i,j\leq 2$ and
$k=1,2$, we have $(u_1^i u_2^j)x^k\in
G\setminus\Gamma$, so
$\hbox{Tr}((u_1^i u_2^j)x^k)=0$, again by Proposition \ref{6.7-Yoshii}. Since $\hbox{tr}$ is $Z$-linear, we are done.

Now setting ${\mathcal J}:=J$, $\mathcal U:=\aa^+$ and $z:=u_3$, we
see that the conditions of \cite[6.14]{Yo1} hold for the mentioned
elements. Therefore $J$ contains a subalgebra $J'$ such that one of
the {following} holds:

(I) there exists a $Z$-isomorphism $\varphi:(\aa,u_3)\longrightarrow
J'$ which acts as identity on $\aa$ and $\varphi((0,1,0))=x$,

(II) there exists a $Z$-isomorphism
$\varphi:(\aa,u_3^{-1})\longrightarrow J'$ which acts as identity on
$\aa$ and $\varphi((0,0,1))=x$.

We assume first that (I) holds and take $\sg\in  G$. Then
$\sg=i\sg_1+j\sg_2+k\sg_3+\gamma$, where $0\leq i,j,k\leq 2$ and
$\gamma\in\Gamma$. Since $\bbba_t$ is of strong type,
$u_0:=t_{\sg_1}^i\cdot(t_{\sg_2}^j\cdot(t_{\sg_3}^k\cdot t_\gamma))$
is a nonzero element of $\bbba_t$ and
$0\not=\varphi(u_0)=u_1^iu_2^jx^ku_\gamma\in J^\a$. Thus $\varphi$ is
an isomorphism over $Z$,  in particular $J\cong_G\bbba_t$.

Next, we assume (II) holds. We note that the $\bbbf$-linear map
$f:\aa=\bigoplus_{\a\in\Delta}\bbbf u_\a\longrightarrow\aa^{{op}}$
induced by $u_1^iu_2^ju_\gamma\longmapsto u_2^iu_1^ju_\gamma$ is an
algebra isomorphism over $\bbbf$. We note
that $f(u_3)=u_3$ and $\hbox{Tr}\circ f=f\circ\hbox{Tr}$. Now the
same reasoning as in the first paragraph of \cite[page 165]{Yo1}
shows that $J\cong_G \bbba_t$. This takes care of the case $\omega\in\bbbf$.

Finally, using an argument analogous to the one given in pages 40-41 of \cite{Yo1}, we see that the case $\omega\not\in\bbbf$ {cannot} happen.
{Thus there is no second Tits construction in our case.}
\qed

\end{document}